\documentclass[12pt,reqno]{amsart}

\usepackage{soul}
\usepackage{amsthm}
\usepackage{amssymb}
\usepackage{latexsym}
\usepackage{multicol,multirow}
\usepackage{verbatim,enumerate}
\usepackage{accents}
\usepackage[usenames]{color}
\usepackage[colorlinks=true, linkcolor=blue, citecolor=green, urlcolor=blue]{hyperref}
\usepackage[nameinlink,noabbrev,capitalize]{cleveref}
\usepackage{nicematrix,makecell}
\usepackage{amsmath, amscd}
\usepackage{soul}
\usepackage{pifont}
\usepackage{dsfont}
\usepackage{dirtytalk,caption,subcaption}
\usepackage{tikz,setspace}
\usepackage[T1]{fontenc}
\usetikzlibrary{matrix,arrows,decorations.pathmorphing}
\usetikzlibrary{positioning}

\advance\textwidth by 1.2in \advance\oddsidemargin by -.6in \advance\evensidemargin by -.6in

\theoremstyle{plain}
\newtheorem{prop}{Proposition}
\newtheorem{thm}{Theorem}
\newtheorem{lem}{Lemma}
\newtheorem{cor}{Corollary}

\theoremstyle{definition}
\newtheorem{example}{Example}
\newtheorem{defn}{Definition}

\theoremstyle{remark}
\newtheorem{rem}{Remark}

\newcommand{\lie}[1]{\mathfrak{#1}}
\newcommand{\wh}[1]{\widehat{#1}}

\newcommand\bc{\mathbb C}

\newcommand\bz{\mathbb Z}
\newcommand\br{\mathbb R}
\newcommand{\ep}{\epsilon}

\newcommand{\cmark}{\ding{51}}%
\newcommand{\xmark}{\ding{55}}%

\def\a{\alpha}

\newcounter{cnt}
 \makeatletter
\def\mydggeometry{\makeatletter\dg@YGRID=1\dg@XGRID=20\unitlength=0.003pt\makeatother}
\makeatother \theoremstyle{remark}

\numberwithin{equation}{section}
\makeatletter
\def\section{\def\@secnumfont{\mdseries}\@startsection{section}{1}%
  \z@{.7\linespacing\@plus\linespacing}{.5\linespacing}%
  {\normalfont\scshape\centering}}
\def\subsection{\def\@secnumfont{\bfseries}\@startsection{subsection}{2}%
  {\parindent}{.5\linespacing\@plus.7\linespacing}{-.5em}%
  {\normalfont\bfseries}}
\makeatother

\begin{document}

\title[]{On symmetric closed subsets of real affine root systems}
\author{Dipnit Biswas}\address{Department of Mathematics, Indian Institute of Science, Bangalore 560012}
\email{dipnitbiswas@iisc.ac.in}
\thanks{}

\author{Irfan Habib}\address{Department of Mathematics, Indian Institute of Science, Bangalore 560012}
\email{irfanhabib@iisc.ac.in}
\thanks{}

\author{R. Venkatesh}
\address{Department of Mathematics, Indian Institute of Science, Bangalore 560012}
\email{rvenkat@iisc.ac.in}
\thanks{}

\subjclass[2010]{}
\begin{abstract}

Any symmetric closed subset of a finite crystallographic root system must be a closed subroot system. This is not, in general, true for real affine root systems. In this paper, we determine when this is true and also give a very explicit description of symmetric closed subsets of real affine root systems. At the end, using our results, we study the correspondence between symmetric closed subsets of real affine root systems and the regular subalgebras generated by them.

\end{abstract}

\maketitle

\section{Introduction}

One of the main motivations to study closed subsets of a given finite crystallographic root system is that they are closely related to the regular subalgebras of the corresponding semi-simple Lie algebra (see \cite{Dynkin, DG20}).
They also appear in various other contexts. For example, they appear in 
\begin{itemize}
    \item the classification of maximal closed connected subgroups of maximal rank of a connected compact Lie group \cite{BdS}
    \item the theory of abelian/ad-nilpotent ideals of Borel subalgebras \cite{Suter, Papi}
   \item the theory Chevalley groups \cite{Harebov}
   \item the classification of reflection subgroups of finite and affine Weyl groups \cite{DL11}.
  \end{itemize}
Various special classes of closed subsets of a finite root system were classified by many authors. For example,
\begin{itemize}
    \item A. Borel and J. De Siebenthal classified maximal closed subroot systems of finite root systems in \cite{BdS}
    \item E. B. Dynkin came up with an algorithm to classify all subroot systems (upto the Weyl group conjugacy) of a finite root system in \cite{Dynkin} 
    \item  Z. I. Borevi{\v{c}} used topological methods in \cite{Bo97} to determine all closed subsets of root systems of type $A_n, 1\le n\le 8$, and this work was subsequently taken forward by others, see for e.g. \cite{BM02}
    \item invertible closed subsets were classified in \cite{DCH94}
    \item recently, A. Douglas and W.A. de Graaf have given an algorithm for classifying all closed subsets of a finite root system, up to conjugation by the associated Weyl group, in \cite{DG20}.
\end{itemize}

The problem of classifying closed subsets of (real) affine root systems is wide open. Again some very particular classes were studied in this setting. 
The invertible closed subsets of affine root systems were classified in \cite{CKS98} and the parabolic subsets were classified in \cite{Futorny}. 
Anna Felikson et al. started the classification regular subalgebras of affine Kac-Moody algebras (which are related to the closed subroot systems of real affine root systems) in \cite{FRT08} and it was completed in \cite{RV19}, see also \cite{DV21}.
The classification of the reflection subgroups of finite and affine Weyl groups has been achieved in \cite{DL11} by classifying 
subroot systems finite and real affine root systems. 
It is well-known that the classification of all subroot systems may be deduced from that of the closed subroot systems in the finite setting, see \cite{Carter72}. A combinatorial description of biclosed sets of real affine root systems has been given very recently in \cite{BS22}. 

We are mainly interested in the closed subsets of real affine root systems in this paper, as they are closely related to the Cartan invariant subalgebras of corresponding affine Kac-Moody algebras and this will be discussed elsewhere.
Suppose $\Psi$ is a closed subset of a real affine root system $\Phi$, i.e., if $\alpha, \beta\in \Psi$ and $\alpha+\beta\in \Phi$ implies that $\alpha+\beta\in \Psi$, then it is easy see that $\Psi$ can be written as a union of its symmetric part $\Psi^r=\Psi\cap -\Psi$ and special part $\Psi^s=\Psi\backslash \Psi^r.$
Both symmetric and special parts of $\Psi$ are closed in $\Phi$. So the classification of closed subsets of $\Phi$ reduces to the classification of 
symmetric and special closed subsets of $\Phi$. In general the problem of classifying special closed subsets is very hard and it will be discussed in \cite{BIV23}.
In this paper we only focus on symmetric closed subsets of real affine root systems.	The main motivation for this work comes from the following two assertions about the finite root system:
	\begin{itemize}
	    \item each symmetric closed subset of a finite root system must be a closed subroot system and
	    \item there is a one to one correspondence between closed subsets of a finite root system and the Cartan invariant subalgebras of corresponding semi-simple Lie algebra (see \cite[Proposition 4.1]{DG20} for the precise statement).
\end{itemize}
Both these statements are not true in general for real affine root systems and this naturally motivates us to ask the following questions:
\begin{enumerate}
	    \item When a given symmetric closed subset of a real affine root system is a closed subroot system?
\item Is it possible to classify all symmetric closed subsets of a real affine root system?
\item  Consider the map $\Psi\mapsto \lie g(\Psi)$ where $\Psi$ is a symmetric closed subset of a real affine root system and  $\lie g(\Psi)$ is the subalgebra generated by $\lie g_\alpha, \,\alpha\in \Psi.$ We will see that this map is not injective for any affine Kac-Moody algebras (see Section \ref{regularsect}). Is it possible to determine the preimage of given $\lie g(\Psi)$?
\end{enumerate} 

We will address all these questions in this paper. The paper is organized as follows: we recall the definitions and set up all the notations in Section \ref{prem}.
The symmetric closed subsets of real affine root systems whose gradient is closed are studied in Section \ref{closedsect}, and
the symmetric closed subsets of real affine root systems whose gradient is semi-closed are studied in Section \ref{semiclosedsect}.
The correspondence between symmetric closed subsets of real affine root systems and regular subalgebras of corresponding affine Lie algebras is discussed in Section \ref{regularsect}.
We summarize all our results at the end, in Section \ref{summarysect}.

\section{Preliminaries}\label{prem}

Throughout this paper, we denote by $\mathbb{C}$ (resp. $\mathbb{R}$), the field of complex numbers (resp. real numbers) and by $\bz$ (resp. $\bz_+$), the set of integers (resp. non-negative integers).

\subsection{}\label{section21}Let $E$ be an Euclidean space over $\mathbb{R}$ endowed with a positive definite symmetric bilinear form  $(\cdot,\cdot)$. A finite (crystallographic) root system
 $\mathring{\Phi}$ is a finite subset of $E$ satisfying the following properties (see \cite[Chapter VI]{Bou46} or \cite[Section 9.2]{Hu80}):
$$0\notin \mathring{\Phi},\ \ \text{Span}_{\mathbb{R}}\mathring{\Phi}=E,\ \ s_{\alpha}(\mathring{\Phi})=\mathring{\Phi},\ \forall\, \alpha\in\mathring{\Phi},\ \ (\beta,\alpha^{\vee})\in\bz,\ \forall\, \alpha,\beta \in \mathring{\Phi},$$
where $\alpha^{\vee}:=2\alpha/(\alpha,\alpha)$ and $s_{\alpha}$ is the reflection acting on $E$ defined by $s_{\alpha}(x)=x-(x,\alpha^{\vee})\alpha$, $x\in E$. For the rest of this paper, we denote $\mathring{\Phi}$ by a finite root system in $E$. In addition, if $\mathring{\Phi}$ satisfies $\br\alpha\cap \mathring{\Phi}=\{\pm \alpha\}$ for $\alpha\in \mathring{\Phi}$, then we call $\mathring{\Phi}$ reduced. Moreover, we call any subset $\Psi\subseteq \mathring{\Phi}$ irreducible whenever $\Psi=\Psi'\cup \Psi''$ with $(\Psi',\Psi'')=0$ implies $\Psi'=\emptyset$ or $\Psi''=\emptyset$. Any root system can be written as a direct sum of irreducible root systems (see \cite[Chapter VI]{Bou46}) and the reduced irreducible root systems were classified in terms of their Dynkin diagrams (see \cite[Theorem 11.4]{Hu80}). 
They are the classical types $A_n\, (n\geq 1), B_n\, (n\geq 2), C_n\, (n\geq 3), D_n\, (n\geq 4)$ and the exceptional types $E_{6,7,8}, F_4$ and $G_2$. For a direct construction of these root systems we refer to \cite[Section 12.1]{Hu80}. Moreover, there is only one non--reduced irreducible root system of rank $n$, namely 
\begin{equation*}BC_n=B_n\cup C_n=\{\pm\epsilon_i: 1\leq i\leq n\}\cup\{\pm\epsilon_i\pm \epsilon_j: 1\leq i\neq j\leq n\}\cup\{\pm 2\epsilon_i: 1\leq i\leq n\},\end{equation*}
where $\epsilon_1,\dots,\epsilon_n$ denotes an orthonormal basis of $E$ with respect to $(\cdot,\cdot)$.
Recall that a subset $\mathring{\Psi}$ of $\mathring{\Phi}$ is said to be symmetric if $\mathring{\Psi}=-\mathring{\Psi}$, and it is called 
closed in $\mathring{\Phi}$ if $\alpha, \beta\in \mathring{\Psi}$ and $\alpha+\beta\in \mathring{\Phi}$ implies $\alpha+\beta\in \mathring{\Psi}$.
We record the following simple and important fact on finite root systems, and we include a proof for the reader's convenience.
	\begin{lem}\label{keylem:finite}
	Let $\mathring{\Psi}$ be a symmetric closed subset of $\mathring{\Phi}$, then $\mathring{\Psi}$ is a closed subroot system of $\mathring{\Phi}$.
	\end{lem}
		\begin{proof}
	Let $\alpha,\beta\in \mathring{\Psi}$. Suppose $(\beta,\alpha^\vee)=0$, then $s_\alpha(\beta)=\beta\in \mathring{\Psi}$. So assume that $(\beta,\alpha^\vee)\neq 0$. 
	If $(\beta,\alpha^\vee)< 0$, then $\beta,\beta+\alpha,\ldots,\beta+(-(\beta,\alpha^\vee))\alpha$ are elements of $\mathring{\Phi}$ by \cite[Proposition 8.4, Page 89]{Hu80}. Since $\mathring{\Psi}$ is closed, we have $\beta,\beta+\alpha,\ldots,\beta+(-(\beta,\alpha^\vee))\alpha\in \mathring{\Psi}$ which implies  $s_\alpha(\beta)=\beta-(\beta,\alpha^\vee)\alpha\in \mathring{\Psi}$.
Since $\mathring{\Psi}$ is symmetric, the case $(\beta,\alpha^\vee)> 0$ can be done similarly.
\end{proof}

\subsection{}
The Weyl group $W$ of a finite root system $\mathring{\Phi}$ is defined to be the subgroup of $GL(E)$ generated by $s_{\alpha},\alpha\in\mathring{\Phi}$. At most two root lengths occur in any  reduced irreducible finite root system $\mathring{\Phi}$  and all roots of a given length are conjugate under the Weyl group $W$ of $\mathring{\Phi}$ (see for example \cite[Section 10.4]{Hu80}). We denote the set of short roots (resp. long roots) by $\mathring{\Phi}_s$ (resp. $\mathring{\Phi}_\ell$) and if there is only one root length then we say that every root is short by convention. If $\mathring{\Phi}$ is non--reduced irreducible finite root system, we define:
$$\mathring{\Phi}_s=\{\pm\epsilon_i,\ 1\leq i\leq n\},\ \ \mathring{\Phi}_\ell=\{\pm\epsilon_i\pm\epsilon_j: 1\leq i\neq j\leq n\},\ \ \mathring{\Phi}_d=\{\pm2\epsilon_i,\ 1\leq i\leq n\}.$$
Note that $\mathring{\Phi}_d=\{\alpha\in \mathring{\Phi}: \alpha/2\in \mathring{\Phi}\}$ is  the set of divisible roots in $BC_n$ and define the non--divisible roots of $BC_n$ by $\mathring{\Phi}_{nd}=\mathring{\Phi}_{}\backslash \mathring{\Phi}_{d}$. Further, set by $\mathring{\Pi}$ a base of $\mathring{\Phi}$ and set
$$ m_{\mathring{\Phi}}=\begin{cases}
1, & \text{$\mathring{\Phi}$ is of type $A_n$, $D_n$ or $E_n$}\\
2, & \text{$\mathring{\Phi}$ is of type $B_n$, $C_n$ or $F_4$}\\
3, & \text{$\mathring{\Phi}$ is of type $G_2$}\\
\end{cases}
$$
We often use $m$ instead of $m_{\mathring{\Phi}}$ if the underlying $\mathring{\Phi}$ is understood.
We end with the following fact (see \cite[Proposition 8.17]{Carter}).
\begin{lem}\label{longroot}
    Let $\mathring{\Phi}$ be a finite irreducible root system.  Let $\beta\in \mathring{\Phi}$ and write $\beta=\sum_{\alpha\in \mathring{\Pi}}k_\alpha \alpha$. Then $\beta$ is a long root if and only if $m$ divides $k_\alpha$ for each short root $\alpha\in \mathring{\Pi}$. \qed
\end{lem}
\medskip
\subsection{}\label{phi}
Let $\mathring{\Phi}$ be an irreducible reduced finite root system. Let $\mathring{\lie g}$ be a finite--dimensional semi-simple Lie algebra over $\mathbb{C}$ and $\mathring{\lie h}$ a Cartan subalgebra of $\mathring{\lie g}$ such that the root system corresponding to the pair $(\mathring{\lie g}, \mathring{\lie h})$ is
  $\mathring{\Phi}$. Let $\sigma$ be a Dynkin diagram automorphism 
of $\mathring{\lie h}$ with respect to $\mathring{\Phi}$ and denote $m$ by the order of $\sigma$. We know that $m\in\{1,2,3\}$. Let $\xi$ be a primitive $m$--th root of unity. We have
$$\mathring{\lie g}=\bigoplus_{j\in \bz/m\bz} \lie g_j,\ \  \lie g_{j}=\{x\in \lie g: \sigma(x)=\xi^j x\}.$$
It is known that $\lie g_0$ is again a finite-dimensional simple Lie algebra over $\mathbb{C}$  with a Cartan subalgebra $\lie h_0=\mathring{\lie h}\cap \lie g_0$. Moreover, $\lie g_j$ is a $\lie g_0$--module and we denote the set of non--zero weights of $\lie g_j$ with respect to $\lie h_0$ by $\mathring{\Phi}_j$.
Then we have $$\mathring{\Phi}:=\mathring{\Phi}_0 \cup\cdots\cup \mathring{\Phi}_{m-1}.$$
The types of $\mathring{\Phi},\mathring{\Phi}_0, \ldots, \mathring{\Phi}_{m-1}$ can be extracted from \cite[Section 7.8, 7.9, 8.3]{K90}. 
The corresponding affine Kac--Moody algebra $\lie g=\widehat{\mathcal{L}}(\mathring{\lie g},\sigma)$ is defined by
$$\lie g=\mathcal{L}(\mathring{\lie g},\sigma)\oplus \bc c\oplus \bc d,\ \ \mathcal{L}(\mathring{\lie g},\sigma)=\bigoplus_{j\in \bz/m\bz} \lie g_{j}\otimes \bc[t^{\pm m}]t^{j},$$ where $\mathcal{L}(\mathring{\lie g},\sigma)$ is called the loop algebra, $\mathcal{L}(\mathring{\lie g},\sigma)\oplus \bc c$ is the universal central extension of the loop algebra and $d=t\frac{d}{dt}$ is the degree derivation. For more details we refer the reader to \cite[Section 7,8]{K90}. The set of roots of $\lie g$ with respect to the Cartan subalgebra $\lie h=\lie h_0\oplus \bc c\oplus \bc d$ is exactly $\Delta\backslash\{0\}$, where
$$\Delta:=\{\alpha+ r\delta: \alpha\in \mathring{\Phi}_j\cup\{0\},\ \ r\in j+m\bz ,\ 0\leq j< m\},$$
Denote $\Phi$ (resp. $\Phi^{\text{im}}$) by the set of real (resp. imaginary) roots of $\lie g$. Then we have $$\Phi=\bigcup_{\alpha\in \mathring{\Phi}}\big(\alpha+ \Lambda_\alpha\big),\ \Phi^{\text{im}}=\bz\delta.$$
where $\Lambda_\a=\Lambda_\beta$ if $\alpha$ and $\beta$ have same length and common $\Lambda_s, \Lambda_\ell, \Lambda_d$ can be found in the following table for each case.
\begin{equation*} \label{n:table1} {\renewcommand{\arraystretch}{1.5}
\begin{array}{|c |c| c| c| c | c| c| c|} \hline

 (\lie g, m) &  \mathring{\Phi}_{0} & \mathring{\Phi}_{1} & \mathring{\Phi}_{2} &\mathring{\Phi} & \Lambda_s &\Lambda_\ell & \Lambda_d \\ \hline

( \lie g(\mathring{\Phi}), 1) &  \mathring{\Phi} & / & / & \mathring{\Phi} & \bz & \bz  & / \\ \hline
  (A_{2n}, 2) &  B_n
    &  \mathring{\Phi}_0 \cup \{\pm 2\epsilon_i : 1 \le i \le n \} & / & BC_n &\frac{1}{2}+\bz& \bz  & 2\bz  \\ \hline
 (A_{2n-1}, 2)   & C_n & (C_{n})_ {s} & / &C_n & \bz & 2\bz  & / \\ \hline
(D_{n+1}, 2) &  B_n & (B_{n})_ {s} & /& B_n & \bz & 2\bz  & / \\
\hline
 (E_6,2) & F_4 & (F_{4})_ {s}  &/ & F_4 & \bz & 2\bz  & / \\ \hline
 (D_4, 3) &  G_2 & (G_{2})_ {s}  & (G_{2})_ {s} & G_2 & \bz & 3\bz  & / \\ \hline
\end{array}}
\end{equation*}

\medskip

We end this section with the following definitions.

\begin{defn}
Let $\Psi$ be a subset of $\Delta$ (resp. $\Phi$). Set $\Psi_{\text{re}}:=\Psi\cap\Phi\  \text{and} \ \Psi_{\text{im}}:=\{r\in\mathbb{Z}:r\delta\in\Psi\}\cup \{0\}.$
\begin{enumerate}
    \item We say $\Psi$ is \textit{symmetric} if $\Psi=-\Psi$, where $-\Psi=\{-\alpha : \alpha\in \Psi\}$.
    \item We say $\Psi$ is \textit{real closed} or \textit{closed} in $\Phi$ if
    \begin{enumerate}
        \item  $\Psi$ is a non-empty subset of $\Phi$
        \item    if $\alpha, \beta\in \Psi$ such that $\alpha+\beta\in \Phi$, then we have $\alpha+\beta\in \Psi$.
    \end{enumerate}
    \item We say $\Psi$ is a \textit{subroot system} if for any $\alpha\in \Psi_{\text{re}}$, $\beta\in \Psi$ we have $s_\alpha(\beta)\in \Psi.$
\end{enumerate}
\end{defn}

\subsection{}
Let $\Phi$ be a real irreducible affine root system as in \Cref{phi} and let $\Psi$ be a closed subset of $\Phi$.
The symmetric part of $\Psi$ defined to be $\Psi^r:=\{\alpha\in\Psi\mid -\alpha\in \Psi\}$ and the special part of $\Psi$ defined to be $\Psi^s:=\{\alpha\in \Psi\mid -\alpha \notin \Psi\}$. It is clear that $\Psi^r$ is a symmetric closed subset of $\Phi$ and $\Psi^s$ is a closed subset of $\Phi$, and we have
$$\Psi=\Psi^r\sqcup \Psi^s.$$ 
The gradient of $\Psi$ is $Gr(\Psi)\,:=\{\alpha\in \mathring{\Phi} \mid \alpha+k\delta\,\in \Psi $ for some $k\in \bz\}$.
	For given $\alpha \in Gr(\Psi)$, we define $Z_\alpha(\Psi)\,:=\{k\in\bz \mid \alpha+k\delta \in \Psi\}$.
We will simply use $Z_{\alpha}$ for $Z_\alpha(\Psi)$ if the dependence of the underlying $\Psi$ is understood.
Clearly we have $$\Psi =\bigcup_{\alpha\in Gr(\Psi)}\{\alpha+k\delta \mid k\in Z_{\alpha}\}.$$
For a given $(p_\alpha)_{\alpha\in Gr(\Psi)}$, where $p_\alpha\in Z_\alpha$, define
$Z_{\alpha}':=Z_\alpha-p_\alpha$ for $\alpha\in Gr(\Psi)$. We often make specific choices of $(p_\alpha)_{\alpha\in Gr(\Psi)}$ and make sure that the map $\alpha\mapsto p_\alpha$ gives us a $\mathbb{Z}$-linear function from $p: Gr(\Psi)\to \mathbb{Z}.$ For example, we have (see \cite[Lemma 13]{DL11a} and \cite[Lemma 2.1.1]{RV19}):
\begin{lem}\label{existp}
Let $\Psi$ be a symmetric closed subset of $\Phi$. Suppose
$Gr(\Psi)$ is a closed, reduced subroot system with a base  $B$ and assume that we have 
\begin{equation}\label{eq:z}
    Z_\alpha+Z_\beta\subseteq Z_{\alpha+\beta}\ \ \text{for all}\ 
    \ (\alpha,\beta,\alpha+\beta)\in Gr(\Psi)^{\times 3}.
\end{equation}
Choose $p_\alpha\in Z_\alpha$ arbitrarily for $\alpha\in B$ and extend the map
$\alpha\mapsto p_\alpha, \alpha\in B$ to $Gr(\Psi)$ $\mathbb{Z}$-linearly. Then we have $p_\alpha\in Z_\alpha$, for all $\alpha\in Gr(\Psi)$.
\end{lem}
\begin{proof}
Define a $\bz$-linear function $p: Gr(\Psi)\to \bz$ extending  $\alpha\mapsto p_\alpha, \alpha\in B$. 
We claim that $p_\alpha\in Z_\alpha$ for all $\alpha\in Gr(\Psi).$ Since $p_{-\alpha}=-p_{\alpha}$ and $Z_{-\alpha}=-Z_\alpha$, it is enough to prove that $p_\alpha\in Z_\alpha$ for all positive roots $\alpha\in Gr(\Psi).$ Let $\alpha\in Gr(\Psi)$ be a positive root, then we can write $\alpha=\alpha_1+\cdots +\alpha_r$ with
each partial sum $\alpha_1+\cdots+\alpha_i$ is again a root in $Gr(\Psi)$ for $1\le i\le r.$ Since $\Psi$ is closed and using the condition \eqref{eq:z}, we have
$\alpha_1+\cdots+\alpha_i+(p_{\alpha_1}+\cdots +p_{\alpha_i})\delta\in \Psi$ for  $1\le i\le r.$ In particular, we have $p_\alpha=p_{\alpha_1}+\cdots +p_{\alpha_r}\in Z_\alpha.$
This completes the proof.
\end{proof}

\subsection{} We collect here some basic facts in this subsection. Let $\Psi$ be a subset of $\Phi$ and let
\begin{equation}\label{decomppsi}
    \Psi=\Psi_1\sqcup\dots\sqcup\Psi_r
\end{equation}
 be the decomposition of $\Psi$ into irreducible subsets.
Then we have 
\begin{equation}\label{decompgrpsi}
    Gr(\Psi)=Gr(\Psi_1)\sqcup\dots \sqcup Gr(\Psi_r)
\end{equation}
 is the decomposition of $Gr(\Psi)$ into irreducible subsets.
\begin{prop}\label{prop:subroot}
Let $\Psi$ be a subset of $\Phi.$  
\begin{enumerate}
     \item Suppose  $\Psi$ is symmetric, then each $\Psi_i$ in \eqref{decomppsi} and each $Gr(\Psi_i)$ in \eqref{decompgrpsi} are symmetric. 
     \item Suppose $\Psi$ is closed in $\Phi$, then each $\Psi_i$ in \eqref{decomppsi} is closed in $\Phi$. In addition if $Gr(\Psi) $ is closed in $\mathring{\Phi}$, then each $Gr(\Psi_i)$ in \eqref{decompgrpsi} is closed in $\mathring{\Phi}$.
     \item  Suppose $\Psi$ is symmetric closed in $\Phi$. Then $\alpha+\beta\notin \Phi$ if $\alpha\in \Psi_i \text{ and } \beta\in \Psi_j$ for $i\neq j$.
     \item Assume that $Gr(\Psi)$ is closed in $\mathring{\Phi}$. Then
     \begin{enumerate}
         \item  $\Psi$ is symmetric closed in $\Phi$ if and only if each $\Psi_i$ is symmetric closed in $\Phi$.
         \item  $\Psi$ is a closed subroot system of $\Phi$ if and only if each $\Psi_i$ is  a closed subroot system of $\Phi$.

     \end{enumerate}

\end{enumerate}
	\end{prop}
	\begin{proof}
	We only prove the statement $4(a)$ as all other statements are easy to check. The forward direction follows from $(3)$. For the converse part, we assume that  each $\Psi_i$ is symmetric closed in $\Phi.$
	Since $Gr(\Psi)$ is symmetric closed in $\mathring{\Phi}$, it must be a closed subroot system. Hence each
	$Gr(\Psi_i)$ is a closed subroot system of $\mathring{\Phi}$ and it has a base say $B_i$. Then $B:=\cup_i B_i$ is a base for $Gr(\Psi)$. Let $\alpha+r\delta\in \Psi_i$ and $\beta+s\delta\in \Psi_j$ with $i\neq j$. If $\alpha+\beta+(r+s)\delta\in \Phi$, then $\alpha+\beta+(r+s)\delta\in \Psi$ since $\Psi$ is closed. But this implies that $\alpha+\beta\in Gr(\Psi)$, which is impossible because any root of a finite root system has connected support, see \cite[Proposition 16.21]{Carter}. This completes the proof.
	\end{proof}
	
	\begin{rem}\label{remark1}
	The fourth statement in \Cref{prop:subroot} is false in general if we drop the condition $Gr(\Psi)$ is closed. For example, let us consider $\Phi$ be of type $D_{n+1}^{(2)}, n
	\ge 2$. Set $I_n=\{1,2,\dots,n\}$.  For $I\subseteq I_n$, let $$\Psi_{I}:=\{\pm \epsilon_k+(2\bz+1)\delta:k\in I\}\cup \{\pm (\epsilon_k\pm \epsilon_\ell)+2\bz\delta:k\neq \ell\in I\}.$$
	It is clear that $\Psi_I$ is a symmetric closed subset of $\Phi$ for each $I.$ Now take $\Psi=\Psi_I\cup \Psi_J$ where $I, J$ form a partition of $I_n$, then $\Psi$ is not a closed subset of $\Phi.$

	\end{rem}

\subsection{}\label{regularsubalgebras}
Given  $S\subseteq \Phi$, we define $\mathfrak{g}(S)$ by the subalgebra of $\mathfrak{g}$ generated by $\mathfrak{g}_\alpha, \alpha\in S.$ Let $\Delta(S)$ be the set of roots of
$\mathfrak{g}(S)$.  
We end this section with the following proposition. 

\begin{prop}\label{smallestclosed}
Let $\Psi$ be a symmetric subset of $\Phi$. Then $\Delta(\Psi)\cap \Phi$ is a minimal closed subroot system of $\Phi$ containing $\Psi.$
\end{prop}
\begin{proof}
    By \cite[Lemma 11.1.2, Page 1301]{RV19}, it is enough to prove that $\Delta(\Psi)=-\Delta(\Psi)$. Suppose $\beta\in \Delta(\Psi)$, then there exists $\beta_1, \ldots, \beta_r\in \Psi$ such that $\beta=\beta_1+\cdots +\beta_r$. 
    Since the Chevalley involution \cite[Chapter 1, Page 7]{K90} of $\mathfrak{g}$ takes $\mathfrak{g}_\alpha$ to $\mathfrak{g}_{-\alpha}$ for all $\alpha\in \Delta$, we have
    $$[\mathfrak{g}_{\alpha_1}, \cdots [\mathfrak{g}_{\alpha_{r-1}}, \mathfrak{g}_{\alpha_r}]] \neq 0 \iff [\mathfrak{g}_{-\alpha_1}, \cdots [\mathfrak{g}_{-\alpha_{r-1}}, \mathfrak{g}_{-\alpha_r}]] \neq 0.$$
     Since $\Psi$ is symmetric, we have  $-\beta_1, \ldots, -\beta_r\in \Psi$
     and $0\neq [\mathfrak{g}_{-\alpha_1}, \cdots [\mathfrak{g}_{-\alpha_{r-1}}, \mathfrak{g}_{-\alpha_r}]]\subseteq \mathfrak{g}(\Psi)$.
    This implies $-\beta\in \Delta$  . Hence we have $\Delta(\Psi)=-\Delta(\Psi).$
    
    Suppose $\Psi\subseteq S$ is a closed subroot system of $\Phi.$ Then we have
    $\mathfrak{g}(\Psi)\subseteq \mathfrak{g}(S)$ and the real root of  $\mathfrak{g}(S)$ is equal to $S$ by \cite[Corollary 11.1.5, Page 1304]{RV19}. Since the real roots of $\mathfrak{g}(\Psi)$ are also real roots of $\mathfrak{g}(S)$, we must have $\Delta(\Psi)\cap \Phi\subseteq S.$ 
\end{proof}

	
\section{Symmetric closed subsets of affine root systems}\label{closedsect}
	In this section, we fix a real affine root system $\Phi$ which is not of type $A_{2n}^{(2)}$. We need the following notation: for $k\in \bz_+,$ denote by $\pi_k:\mathbb{Z}\to \mathbb{Z}/k\mathbb{Z}$ the qutioent map $x\to x\ (\mathrm{mod}\ k)$. Recall that $m$ is the lacing number associated with $\mathring{\Phi}.$

\subsection{}  We define semi-closed subset of finite root systems as in \cite[Definition 4.4.1]{RV19}.
\begin{defn}
	A symmetric subset $\mathring{\Psi}$ of $\mathring{\Phi}$ is called semi-closed subset of $\mathring{\Phi}$ if 
	\begin{enumerate}
	    \item $\mathring{\Psi}$ is not closed in $\mathring{\Phi}$.
	    \item If $\alpha,\beta\in \mathring{\Psi}$ such that $\alpha+\beta\in \mathring{\Phi}\setminus \mathring{\Psi}$, then 
	    $(\alpha, \beta, \alpha+\beta)$ is of type $(s,s,\ell)$.
	\end{enumerate}
	\end{defn}
We need the following simple lemma, see \cite[Proposition 4.1.2]{RV19}. 
	\begin{lem}\label{keylem:affine}
		Let $\Phi$ be a real twisted affine root system not of type $A_{2n}^{(2)}$. Let $\Psi\leq \Phi$ be a closed subset. Then $Gr(\Psi)$ is either closed or semi-closed subset of $\mathring{\Phi}$. \qed
	\end{lem}

	\subsection{}
	We now consider the case when $Gr(\Psi)$ is a closed subroot system of $\mathring{\Phi}$. 
	The following proposition will be used as a primary tool to prove our main theorem in most of the cases.
\begin{prop}\label{prop:1}
Let $\Phi$ be a real affine root system not of type $A_{2n}^{(2)}$. 
Let $\Psi$ be a symmetric closed subset of $\Phi$ such that $Gr(\Psi)$ is closed in $\mathring{\Phi}$ and  no irreducible component of $Gr(\Psi)$ is of type $A_1$. Suppose that there is  a $\bz$-linear function $p:Gr(\Psi)\to\bz$, $\alpha\mapsto p_\alpha$, such that $p_\alpha\in Z_\alpha,$\ and $|\pi_m(Z_\alpha)|=1, \text{for all} \ \alpha\in Gr(\Psi)$
Then we have
\begin{enumerate}
\item $Z_\alpha'=Z_\beta'=Z_{\alpha+\beta}':=A$  for 
$(\alpha,\beta, \alpha+\beta)\in Gr(\Psi)^{\times 3}$ and $A$ is a subgroup of $m\bz$. \item For any $\alpha,\beta\in\Psi$, we have $s_\alpha(\beta)\in\Psi$. In particular, $\Psi$ is a closed subroot system of $\Phi$.
\end{enumerate}
\end{prop}
\begin{proof}
Recall that $Z_\alpha'=Z_\alpha-p_\alpha$ for $\alpha\in Gr(\Psi).$
Let $(\alpha, \beta, \alpha+\beta)\in Gr(\Psi)^{\times 3}$. We claim that $$Z_\alpha+Z_\beta\subseteq Z_{\alpha+\beta}.$$Let $r\in Z_\alpha$ and $s\in Z_\beta$. If $\alpha+\beta$ is short or $\Phi$ is untwisted then the result is immediate. So assume that $\alpha+\beta$ is long. Then we have
$p_{\alpha+\beta}\equiv 0 \,(\mathrm{mod}\,m)$.
 Since  $r\equiv p_\alpha\,(\mathrm{mod}\, m)\text{ and } s\equiv p_\beta\,(\mathrm{mod}\, m)$ and the map $p$ is $\bz$-linear we have $r+s\equiv 0 \,(\mathrm{mod}\,m)$.
 This implies $\alpha+\beta+(r+s)\delta\in \Phi$. Since $\Psi$ is closed, it follows that $\alpha+\beta+(r+s)\delta\in \Psi$ and so $r+s\in Z_{\alpha+\beta}$.
 Hence $Z_\alpha+Z_\beta\subseteq Z_{\alpha+\beta}.$ Now using the linearity of $p$, we get
 \begin{equation}\label{keyfact1}
     Z_\alpha'+Z_\beta'\subseteq Z_{\alpha+\beta}'
 \end{equation}
for all $(\alpha, \beta, \alpha+\beta)\in Gr(\Psi)^{\times 3}$. 
Since \eqref{keyfact1} is true for all 
tuples $(\alpha,\beta, \alpha+\beta)$ such that $\alpha, \beta, \alpha+\beta\in Gr(\Psi)$.
We have
$$ Z_\alpha'+Z_\beta' \subseteq Z_{\alpha+\beta}',\ \   Z_{\alpha+\beta}'+Z_{-\alpha}'\subseteq Z_\beta',\ \text{and} \
   Z_{\alpha+\beta}'+Z_{-\beta}' \subseteq Z_\alpha'.
$$ This implies $A=Z_\alpha'=Z_\beta'=Z_{\alpha+\beta}'$.
 It is easy to see that $Z_{-\mu}'=-Z_{\mu}'$ for all $\mu\in$ Gr$(\Psi)$. Thus $A\subseteq m\bz$ satisfies $A=-A$ and $A+A\subseteq A$, so it must be a subgroup of $m\bz$.

\medskip
To prove $(2)$, let $\alpha, \beta \in \Psi.$ Write $\alpha=\alpha'+s\delta,\, \beta=\beta'+r\delta$. We claim that
$s_{\alpha}(\beta)=\beta-(\beta,\alpha^\vee)\alpha\in \Psi$. Suppose $\beta'\neq \pm \alpha'$, then using \say{unbroken string property} of $\Delta$ we get
$s_{\alpha}(\beta)\in \Psi$ as $\Psi$ is closed and any root of the form $\beta+p\alpha, p\in \bz,$ must be real. 
Now assume that $\beta'= \pm \alpha'$. Since no irreducible component of $Gr(\Psi)$ is of type $A_1$, there is a $\gamma'\in Gr(\Psi)$  such that $\alpha'+\gamma'\in$ Gr$(\Psi)$. 
Applying the Part (1) for the triple $(\alpha',\gamma',\alpha'+\gamma')$, we get that $Z_{\alpha'}'$ is a subgroup of $m\bz$.
We have
$$s_{\alpha'+s\delta}(\pm\alpha'+r\delta)=\mp\alpha'+(r\mp 2s)\delta$$
Since $s\in Z_{\alpha'},\,r\in Z_{\pm\alpha'}$, there are $z_1,z_2\in Z_{\alpha'}'$ such that $s=z_1+p_{\alpha'}$ and $r=z_2+ p_{\pm \alpha'}$.
Now we have $r\mp 2s=z_2\pm p_{\alpha'}\mp 2(z_1+p_{\alpha'})=(z_2\mp 2z_1)\mp p_{\alpha'}\in Z_{\alpha'}'+p_{\mp \alpha'}$ as $Z_{\alpha'}'$ is a group. 
Hence $s_{\alpha}(\beta)\in \Psi$ and this completes the proof.
\end{proof}

\subsection{}\label{closedcase} When $Gr(\Psi)$ is closed in $\mathring{\Phi}$,  using \Cref{prop:subroot} 4(a), we see that to classify a symmetric closed subsets of $\Phi$ we only need to classify irreducible symmetric closed subsets of $\Phi.$ 
So without loss of any generality we assume that $\Psi$ is irreducible symmetric and closed in $\Phi$ in what follows.


\begin{prop}\label{prop:equal1}
Let $\Phi$ be a real affine root system not of type $A_{2n}^{(2)}$.	    
Let $\Psi$ be an irreducible symmetric closed subset of $\Phi$ such that $Gr(\Psi)$ is closed in $\mathring{\Phi}$ and it is not of type $A_1$. If there exists a short root $\beta$ such that 
$|\pi_m(Z_\beta)|=1$, then
there exists $n\in \bz_+$ and a $\bz$-linear function $p:Gr(\Psi)\to\bz$, $\alpha\mapsto p_\alpha$, with $p_\alpha\in Z_\alpha,\, \text{for all} \ \alpha\in Gr(\Psi)$ such that $$\Psi=\{\alpha+(p_\alpha+n\bz)\delta:\alpha\in Gr(\Psi)\}$$
In particular $\Psi$ is a closed subroot system of $\Phi$.	\end{prop}
	\begin{proof}
	Note that $Gr(\Psi)$ must be a closed subroot system of $\mathring{\Phi}.$
	Fix a short simple root $\beta\in Gr(\Psi)$ such that $|\pi_m(Z_\beta)|=1$. We first show that $|\pi_m(Z_\alpha)|=1$ for all simple roots $\alpha$. We prove this by induction on $k$ where $\beta=\alpha_0,\cdots,\alpha_k=\alpha$ is the unique path from $\beta$ to $\alpha$.
	Set $\gamma_j=\sum_{i=0}^{j}\alpha_{i}$ for $0\le j\le k$, and note that  $\gamma_j\in Gr(\Psi)$, for all $0\le j\le k$.
We claim that  $|\pi_m(Z_{\alpha_j})|=1$ and $|\pi_m(Z_{\gamma_j})|=1$, for all $0\le j\le k$. 
	For $k=1$, we have $\alpha$ is adjacent to 
	 $\beta$ and $\alpha+\beta\in Gr(\Psi)$.
Consider $Z_{\alpha+\beta}+Z_{-\alpha}\subseteq Z_\beta$ which is always true as $\beta$ is short. This gives $Z_{\alpha+\beta}$ and
$Z_{-\alpha}$ must contain only one congruence class modulo $m.$ Since $Z_{-\alpha}=-Z_\alpha$, we have $|\pi_m(Z_\gamma)|=1$ for $\gamma=\alpha, \alpha+\beta.$
For general $k$, we must have $\gamma_{k-1}$ is short by \Cref{longroot} since $\beta$ is short. Since $\gamma_k=\gamma_{k-1}+\alpha$, we have $Z_{\gamma_k}+Z_{-\alpha}\subseteq Z_{\gamma_{k-1}}$. By induction hypothesis, we have $|\pi_m(Z_{\gamma_{k-1}})|=1$, and this implies the desired claim.

We shall now show that $|\pi_m(Z_\alpha)|=1$ for every positive root $\alpha.$ We proceed by induction on ht$(\alpha)$.
If $\alpha$ is long or ht$(\alpha)=1$, then there is nothing to prove.
 Assume that ht$(\alpha)>1$, then we can write $\alpha=\gamma+\alpha_k$ where $\gamma\in \mathring{\Phi}^+ $ and $\alpha_k$ is a simple root. 
 If $\alpha_k$ is short, we have $Z_{\gamma}+Z_{-\alpha}\subseteq Z_{-\alpha_k}$ and this immediately implies that $|\pi_m(Z_\alpha)|=1$.
 If $\alpha_k$ is long, then $\gamma$ is short and we have $Z_{\alpha}+Z_{-\alpha_k}\subseteq Z_{\gamma}$.  By induction hypothesis, we have $|\pi_m(Z_\gamma)|=1$, hence we get $|\pi_m(Z_\alpha)|=1$. We now check the condition \eqref{eq:z}
    is satisfied.  Since $\Psi$ is closed, the only non-trivial case is when $\Phi$ is twisted and 
    $(\alpha_1,\alpha_2,\alpha_1+\alpha_2)\in Gr(\Psi)^{\times 3}$ is of type $(s,s,\ell)$.   
  Since $Z_{\alpha_1+\alpha_2}+Z_{-\alpha_1}\subseteq Z_{\alpha_2}$ holds, we must have $r+s\equiv 0\ (\mathrm{mod}\ m)$ for all  $r\in Z_{\alpha_1}$ and $s\in Z_{\alpha_2}$.
  This implies $Z_{\alpha_1}+Z_{\alpha_2}\subseteq Z_{\alpha_1+\alpha_2}$. Now \Cref{prop:1} completes the proof.

	\end{proof}
	
	\begin{cor}\label{coruntwisted}
    Let $\Phi$ be a real untwisted irreducible affine root system.
	Let $\Psi$ be a symmetric closed subset of $\Phi$ such that none of the irreducible components of $Gr(\Psi)$ is of type $A_1$. Then $\Psi$ is a closed subroot system of $\Phi$. 
	\end{cor}
	\begin{proof}
	Since $m=1,$ we get the result from \cref{prop:equal1}.
	\end{proof}
	
\begin{rem}
The assumptions on the $Gr(\Psi)$ in \cref{prop:1} and \ref{prop:equal1} and \Cref{coruntwisted} are very sharp, i.e.,  the conclusions of \cref{prop:1} and \ref{prop:equal1} and  \Cref{coruntwisted}  are not valid when one of the components of $Gr(\Psi)$ is of type $A_1$. For example, we can take $$\Psi=\{\alpha+\delta,-\alpha+\delta,-\alpha-\delta,\alpha-\delta\}$$ in $A_1^{(1)}$, it is symmetric closed but not subroot system. 
Take $\alpha_1=\alpha+\delta,\,\beta_1=-\alpha+\delta$, then $s_{\alpha_1}(\beta_1)=\alpha+3\delta\notin \Psi$. 
Note that $Z_\alpha(\Psi)=\{\pm 1\}$ which is far from being a coset. 

\end{rem}

\subsubsection{}
Now we assume that there is a short root $\beta\in Gr(\Psi)$ such that $|\pi_m(Z_\beta)|>1$. In this case, we have to deal with the case $Gr(\Psi)$ is of type $B_2$  separately. First we assume that $Gr(\Psi)$ is not of type $B_2$, then we have:
\begin{prop}\label{prop:notb2}
Let $\Phi$, $\Psi$, $Gr(\Psi)$ be as before in \Cref{prop:equal1}.
Further assume that $Gr(\Psi)$ is not of type $B_2$. Suppose that there exists a short root $\beta\in Gr(\Psi)$ such that $|\pi_m(Z_\beta)|>1$,
	 then we have,
	     \begin{equation}\label{notb2}
	         \Psi=\{\alpha+(p_\alpha+n_s\bz)\delta:\alpha\in Gr(\Psi)_s\}\cup \{\gamma+(p_\gamma+mn_s\bz)\delta:\gamma\in Gr(\Psi)_\ell\}
	     \end{equation}
	for some $n_s\in \bz_+$. In particular $\Psi$ is a closed subroot system.
    \end{prop}
    
    \begin{proof}
      Clearly $\Phi$ is twisted and we have $|\pi_m(Z_\alpha)|>1$ for all short roots $\alpha\in \Pi$ by the proof of  \Cref{prop:equal1}. Let $Z_{\gamma, 0}=Z_\gamma\cap m\bz,$ for $\gamma\in Gr(\Psi)$ and note that $Z_{\gamma,0}=Z_\gamma$ for long roots $\gamma\in Gr(\Psi)$.  Define $\Psi_0$ as follows
    $$\Psi_0:=\{\gamma+Z_{\gamma, 0}\delta:\gamma\in Gr(\Psi)_s\}\cup \{\gamma+Z_{\gamma,0}\delta:\gamma\in Gr(\Psi)_\ell\}$$
    If $m=2$, clearly $Z_{\gamma, 0}\neq \emptyset$ for all $\gamma\in Gr(\Psi)$. If $m=3$, then $\Phi=D_4^{(3)}$. Since $Gr(\Psi)$ is not of type $A_1$ and it is irreducible, we must have $Gr(\Psi)=\mathring{\Phi}$ (which is of type $G_2$). If $\{\alpha_1,\alpha_2\}$ is a basis of $\mathring{\Phi}$ with $\alpha_1$ is short, then 
    the short roots of $\mathring{\Phi}$ are $\alpha_1, \alpha_2+\alpha_1$ and $\alpha_2+2\alpha_1$.
Since $Z_{\alpha_2}\subseteq m\bz$ and $Z_{\alpha_2}+Z_{\alpha_1}\subseteq Z_{\alpha_2+\alpha_1}$, we have $$\pi_m(Z_{\alpha_1})\subseteq \pi_m(Z_{\alpha_2+\alpha_1}).$$ Let $x, y\in \pi_m(Z_{\alpha_1})$ such that $x\neq y.$
Since $Z_{\alpha_2+\alpha_1}+Z_{\alpha_1}\subseteq Z_{\alpha_2+2\alpha_1}$, we have
$2x, x+y, 2y\in \pi_m(Z_{\alpha_2+2\alpha_1}).$ Since $x\neq y$, we have
$|\pi_m(Z_{\alpha_2+2\alpha_1})| = 3$. Now since 
    $Z_{\alpha_2+2\alpha_1}+Z_{-\alpha_1}\subseteq Z_{\alpha_2+\alpha_1}$, we have
    $|\pi_m(Z_{\alpha_2+\alpha_1})| = 3$.
    Again since $Z_{\alpha_2+\alpha_1}+Z_{-\alpha_2}\subseteq Z_{\alpha_1}$, we must have
     $|\pi_m(Z_{\alpha_1})| = 3$.
This proves that $Z_{\alpha,0}\neq \emptyset$ for $\alpha\in \mathring{\Phi}_s$.   Hence $Z_{\gamma,0}\neq \emptyset$ for all $\gamma\in Gr(\Psi)$.

\medskip   
   Clearly the condition \eqref{eq:z} is satisfied for $\Psi_0$. Hence we have a $\bz$-linear function $p:Gr(\Psi_0)\to \bz$ with $p_\gamma\in Z_{\gamma,0}$ for all $\gamma\in Gr(\Psi_0)=Gr(\Psi)$. All the necessary conditions of \Cref{prop:1} are satisfied for $\Psi_0$. Since $Gr(\Psi)$ is irreducible, there exists $n_\ell\in m\bz$ such that
   $Z'_{\gamma, 0}=n_\ell \bz$ for all $\gamma\in Gr(\Psi)$. In particular, we have
   $Z'_{\gamma}=n_\ell \bz$ for all long roots $\gamma\in Gr(\Psi)$.
   
   \medskip
      Now let $Gr(\Psi)$ be of type $B_n$ with $n\geq 2$, in particular $m=2$. We claim that $Z_\alpha'$ is a union on cosets modulo $n_\ell\bz$ for all $\alpha\in Gr(\Psi)_s$.  Let $\alpha\in Gr(\Psi)$ be a short root. Since $Gr(\Psi)$ is irreducible, there exists a long root $\gamma$ such that $(\alpha,\gamma)\neq 0$.  Without loss of generality let $(\alpha,\gamma)<0$. Since $\Psi$ is closed, we have the following relations 
    $$Z_\gamma+Z_\alpha\subseteq Z_{\alpha+\gamma},\ Z_{\alpha+\gamma}+Z_{-\gamma}\subseteq Z_{\alpha}.$$
Since $p$ is $\bz$-linear, we have 
    $$Z_\gamma'+Z_\alpha'\subseteq Z_{\alpha+\gamma}',\ \ Z_{\alpha+\gamma}'+Z_{-\gamma}'\subseteq Z_{\alpha}'.$$
    Since $Z_\gamma'=n_\ell \bz$, the above relations show that $Z_\alpha'=Z_{\alpha+\gamma}'$ and $Z_\alpha'+n_\ell\bz\subseteq Z_\alpha'$. Thus we have 
    $Z_\alpha'$ is a union of cosets modulo $n_\ell\bz$ for all short roots $\alpha\in Gr(\Psi)$. 
    Now if $\alpha$ and $\gamma$ are short roots such that $\alpha+\gamma$ is long, then we have
    \begin{equation}\label{keyeq2}
    (Z_\alpha'+Z_\gamma')\cap m\bz\subseteq Z_{\alpha+\gamma}',\ \ Z_{\alpha+\gamma}'+Z_{-\gamma}'\subseteq Z_{\alpha}'
    \end{equation}
Fix a short root $\alpha$ and write $Z_\alpha'=\cup_i(a^i_{\alpha}+n_\ell\bz)$ with $0\le a^i_{\alpha}\neq a^j_{\alpha}< n_\ell,\ i\neq j$. Since $Gr(\Psi)$ is of type $B_n\ (n\geq 2)$, there is a short root $\gamma$ such that $\alpha+\gamma$ is long. Equations in \eqref{keyeq2} imply that $n_\ell\bz\subseteq Z_{\alpha}'$ and hence $a^k_\alpha=0$ for some $k$ and $2\nmid a^j_\alpha$ for all $j\neq k$. Since $|\pi_m(Z_\alpha)|>1,$ there exists $i$ such that $a_\alpha^i$ is odd. Let $a_\gamma^r$ be one such odd coset for $\gamma$. If $a^i_\alpha,a^j_\alpha$ are both odd, then by equation \eqref{keyeq2}, we get that $a^i_\alpha+a^r_\gamma\equiv 0\ (\mathrm{mod}\ n_\ell)$ and $a^j_\alpha+a^r_\gamma\equiv 0\ (\mathrm{mod}\ n_\ell)$. Hence it follows that $a^i_\alpha=a^j_\alpha$ and $Z_\alpha'$ is a union of exactly two cosets $n_\ell\bz\cup (a_\alpha+n_\ell\bz)$ with $a_\alpha$ odd. Similarly $Z_\gamma'=n_\ell\bz\cup (a_\gamma+n_\ell\bz)$. Furthermore $a_\alpha+a_\gamma\equiv 0(\mathrm{mod}\ n_\ell)$. Till this point the argument is valid for $B_2$ as well.
    Now assume that $n\geq 3$. Let $\beta_1,\beta_2$ be short roots. Then there is a short root $\beta_3$ such that $\beta_i+\beta_j$ is a long root for $1\le i\neq j\le 3$. Applying equation \eqref{keyeq2} for the triples $(\beta_1,\beta_2,\beta_1+\beta_2),\ (\beta_2,\beta_3,\beta_2+\beta_3),\ (\beta_1,\beta_3,\beta_1+\beta_3)$ successively we get that $$a_{\beta_1}+a_{\beta_2}\equiv 0\ (\mathrm{mod}\ n_\ell),\ \ a_{\beta_2}+a_{\beta_3}\equiv 0\ (\mathrm{mod}\ n_\ell),\ \ a_{\beta_1}+a_{\beta_3}\equiv 0\ (\mathrm{mod}\ n_\ell)$$ 
    Hence we have $a_{\beta_1}=a_{\beta_2}=a_{\beta_3}=n_\ell/2$. This proves that $Z_{\beta_1}'=Z_{\beta_2}'=n_\ell\bz\cup (n_\ell/2+n_\ell\bz)=\frac{n_\ell}{2}\bz.$ Hence if $Gr(\Psi)$ is of type $B_n(n\geq 3),$ then $\Psi$ is of the form \eqref{notb2}.
    
    \medskip
     Now let $Gr(\Psi)$ be not of type $B_n$. Let $\Psi'=\bigcup_{\alpha\in \mathring{\Phi}_s}\{\alpha+r\delta\in \Psi:r\in \bz\}$. Since $Gr(\Psi)_s=Gr(\Psi')$, we have
     $\Psi'$ is symmetric closed subset of $\mathring{\Phi}_s^{(1)}$. Note that $Z_\alpha(\Psi)=Z_\alpha(\Psi')$ for all $\alpha\in Gr(\Psi)_s.$
     Since $Gr(\Psi)_s$ is irreducible, by \Cref{prop:1}, there exists $n_s\in \bz_+$ such that $Z_\alpha'(\Psi)=Z_\gamma'(\Psi)=n_s\bz$ for all $\alpha,\gamma\in Gr(\Psi)_s$.
     Note that given $\alpha\in Gr(\Psi)_s$, there exists a $\gamma\in Gr(\Psi)_s$ such that $\alpha+\gamma\in Gr(\Psi)_\ell$.  Applying equation \eqref{keyeq2}, we get that $n_\ell\bz\subseteq n_s\bz$ and $mn_s\bz=n_s\bz\cap m\bz\subseteq n_\ell\bz$. We have $n_s\mid n_\ell$ and $n_\ell\mathrel{|} mn_s$. Let $n_\ell=rn_s$. Then $r\mid m$. Since $m$ is prime, either $r=1$ or $r=m$. If $r=1,$ then $n_\ell=n_s$, and since $m\mid n_\ell$, we have that $|\pi_m(Z_\alpha)|=1$ for all short roots $\alpha$ which is a contradiction. 
     Hence $r=m$ and consequently, $\Psi$ is of the form \eqref{notb2}. This completes the proof.

    \end{proof}

    \begin{prop}\label{prop:b2}
    Let $\Phi$, $\Psi$, $Gr(\Psi)$ as before in \Cref{prop:equal1}.
   Assume that there exists a short root $\beta\in Gr(\Psi)$ such that $|\pi_m( Z_\beta)|>1$. Suppose $Gr(\Psi)$ is of type $B_2$, then we have $\Psi=\Psi^+\cup (-\Psi^+)$ where $\Psi^+$ is of the form 
	 \begin{equation}\label{b2}
 \Psi^+=\{\epsilon_i+ (p_{\epsilon_i}+A_i)\delta: i=1, 2\}\cup\{\alpha+(p_\alpha+n_\ell\bz)\delta:\alpha\in \mathring{\Phi}^+_\ell\}
    \end{equation}
    where 
    \begin{enumerate}
        \item $p:Gr(\Psi)\to \bz$ is a $\bz$-linear function such that $p_{\epsilon_i}\in 2\bz$ for $i=1,2$,
        \item $A_i=n_\ell\bz \cup (a_i+n_\ell\bz),$
       $n_\ell\in 2\bz_+$, $0\le a_1,a_2< n_\ell$ such that $a_1+a_2\equiv 0\ (\mathrm{mod}\ n_\ell)$ and both $a_1, a_2$ are odd.
    \end{enumerate}
      Moreover, we have $\Psi$ is a closed subroot system if and only if $a_1=a_2=n_\ell/2$. In this case $\Psi$ is of the form
      \begin{equation}\label{b2subroot}
          \Psi=\{\alpha+(p_\alpha+n_s\bz)\delta:\alpha\in Gr(\Psi)_s\}\cup \{\beta+(p_\beta+mn_s\bz)\delta:\beta\in Gr(\Psi)_\ell\}
      \end{equation}
    \end{prop}

    \begin{proof}
       Note that in this case, from the proof of \Cref{prop:notb2},  $Gr(\Psi)$ is of the form \eqref{b2}. If $a_1\equiv a_2(\mathrm{mod}\ n_\ell),$ then along with $a_1+a_2\equiv 0(\mathrm{mod}\ n_\ell)$, we get that $a_1=a_2=n_\ell/2$. Hence $(n_\ell\bz)\cup (a_2+n_\ell\bz)=(n_\ell/2)\bz$ and hence $\Psi$ is of the form \eqref{notb2}. In particular, $\Psi$ is a closed subroot system.
    
    Conversely, suppose that $Gr(\Psi)$ is of type $B_2$ and $\Psi$ given in \eqref{b2} is a closed subroot system of $\Phi$. Let $\alpha=\epsilon_i+(p_{\epsilon_i}+a_i)\delta, \ \alpha'=\epsilon_i+p_{\epsilon_i}\delta$. Then $s_{\alpha'}(\alpha)=-\epsilon_i+(p_{-\epsilon_i}+a_i)\delta$. Since $p_{-\epsilon_i}+a_i$ is odd and $p_{-\epsilon_i}+(-a_i)+n_\ell\bz$ is the coset in $Z_{-\epsilon_i}$ which contains odd integers, we must have $a_i\equiv -a_i(\mathrm{mod}\ n_\ell)$. Hence $a_i=n_\ell/2$ for $i=1,2.$
    
    \end{proof}

    \section{Semi-closed gradient case}\label{semiclosedsect}
As before, we assume that $\Phi$ is a real twisted affine root system which is not of type $A^{(2)}_{2n}$. In this section,
we consider the case when $\Psi$ is a symmetric closed subset of $\Phi$ with $Gr(\Psi)$ is semi-closed in $\mathring{\Phi}.$ We cannot assume that $\Psi$ is irreducible as in \cref{closedcase}, because we may miss the interaction between two roots coming from the different orthogonal components of \eqref{decompgrpsi} (see \Cref{remark1}). Moreover, we have to deal with each individual affine root system separately due to their gradient behaviour. 
    
    \subsection{Twisted real affine root system of type \texorpdfstring{$D_{n+1}^{(2)}$}{D(2)n+1}}\label{Dn+1}  In this case, the real affine roots are given by
$$\Phi=\{\pm \epsilon_i\,+r\,\delta\,\mid r\in\mathbb{Z},1\le i\le n\}\cup\,\{\pm\epsilon_i\pm\epsilon_j\,+2r\delta\,\mid r\in \mathbb{Z},\,1\le i\neq j\le n\}$$
and $\mathring{\Phi}$ is of type $B_n, n\ge 2$. We assume that $\Psi$ is a symmetric closed subset of $\Phi$ such that $Gr(\Psi)$ is semi-closed in $\mathring{\Phi}$. For $J\subseteq I_n$, we define $$B_{J}:=\{\pm\epsilon_{j}:j\in J\}\cup \{\pm\epsilon_j\pm \epsilon_{j'}: j\neq j'\in J\}.$$ 

\noindent
Recall that by definition, a semi-closed subset is symmetric.
First we will determine the semi-closed subsets of $\mathring{\Phi}$ which occur as the gradient of some symmetric closed subset of $D_{n+1}^{(2)}$. Maximal semi-closed subroot systems of $\mathring{\Phi}$ were determined in \cite{DV21}.  The following result generalizes \cite[Theorem 2(2)]{DV21} for $B_n$ type.

\begin{lem}\label{semiclosedD}
Let $\Phi$ be of type $D_{n+1}^{(2)}$ and $\Psi$ a symmetric closed subset of $\Phi$.  Suppose that $\mathring{
\Psi}$ is a semi-closed subset of $\mathring{\Phi}$ such that $\mathring{\Psi}=Gr(\Psi)$, then there exist non-empty subsets $J_o$ and $J_e$ of $I_n$ such that $J_e\cap J_o=\emptyset$ and $\mathring{\Psi}\backslash (B_{J_e}\cup B_{J_o})$ is a closed subroot system of long roots of $B_n$.
\end{lem}
\begin{proof}
Assume that $\mathring{\Psi}=Gr(\Psi)$ for some symmetric closed subset $\Psi$ of $D_{n+1}^{(2)}$. 	Set $Z_p=Z_{\epsilon_p}$ for each $p\in I$ where $$I:=\{p\in I_n\mid \epsilon_p\in \mathring{\Psi}\}.$$ Note that $\epsilon_p+\epsilon_q\in \mathring{\Psi}$ for $p, q\in I$ if and only if $(Z_p+Z_q)\cap 2\bz\neq \emptyset$.
Since $Z_{-\epsilon_p}=-Z_p$, we have $\epsilon_p+\epsilon_q\in \mathring{\Psi}$ for $p, q\in I$ if and only if $\epsilon_p-\epsilon_q\in \mathring{\Psi}$ for $p, q\in I$.
Since $\mathring{\Psi}$ is semi-closed, there exist $\epsilon_i,\epsilon_j\in \mathring{\Psi}$ such that $\epsilon_i+\epsilon_j\notin \mathring{\Psi}$, this implies $I\neq \emptyset$ and
 $(Z_i+Z_j)\cap 2\bz=\emptyset$. Without loss of generality, assume that $Z_i\subseteq 2\mathbb{Z}$ and $Z_j\subseteq 2\mathbb{Z}+1$. 
 We claim that for any $k\in I$, we have either $Z_k\subseteq 2\bz$ or $Z_k\subseteq 2\bz+1$. Suppose that for some $k\in I$, $Z_k$ contains both even and odd elements. Then $\exists \,r,s\in\bz$ such that $\ep_k+2r\delta,\ep_k+(2s+1)\delta\in\Psi$. Let $2q\in Z_i$ and $2p+1\in Z_j$ be arbitrary. Then we have
 $\ep_k+\ep_i+2(r+q)\delta, \ep_k-\ep_j+2(s-p)\delta\in \Psi$ and
		$$\ep_k+\ep_i+2(r+q)\delta-(\ep_k-\ep_j+2(s-p)\delta)= \ep_i+\ep_j+2(r+q+p-s)\delta\in \Psi$$ as $\Psi$ is closed in $\Phi$. This implies 
		$\ep_i+\ep_j\in \mathring{\Psi}$, a contradiction. Hence the claim.
		
		Now define  $J_e:=\{ p\in I\,\mid Z_p\subseteq 2\bz\}$ and $J_o:=I\backslash J_e$.
		We claim that $(s, t), (t, s)\in J_e\times J_o$ implies that $\pm\ep_s\pm \ep_t\notin \mathring{\Psi}$.
		Suppose that $\pm \ep_s\pm \ep_t+2k\delta\in \Psi$ with $(s, t)\in J_e\times J_o$. Then we have $$\pm\ep_t+2(k+r^{\mp})\delta\,=(\pm \ep_s\pm \ep_t+2k\delta)+(\mp \ep_s+2r^{\mp}\delta)\in \Psi,$$
		where $r^{\mp}\in\bz$ such that  $\mp \ep_s+2r^{\mp}\delta\in \Psi$. This is a contradiction since $t\in J_o$. By symmetry we get $(t, s)\in J_e\times J_o$ implies that $\pm\ep_s\pm \ep_t\notin \mathring{\Psi}$.
		Note that $\pm\epsilon_i\pm \epsilon_j\in \mathring{\Psi}$ for all $i,j\in J_e\ $(or $J_o$). Hence $\{\pm \epsilon_i: i\in I\}$ generates a root system of type $B_{J_e}\sqcup B_{J_o}$ in $\mathring{\Psi}$. If $\mathring{\Psi}=B_{J_e}\sqcup B_{J_o}$, then there is nothing to prove.  So assume that 
		 $\mathring{\Psi}\neq B_{J_e}\sqcup B_{J_o}$. It is clear that $\mathring{\Psi}\backslash (B_{J_e}\sqcup B_{J_o})$ is a symmetric closed subset of $(B_n)_\ell$ since 
		 $\mathring{\Psi}$ is semi-closed in $\mathring{\Phi}.$ Hence it is a closed subroot system of $(B_n)_\ell$ by \Cref{keylem:finite}. 
		 \end{proof}
		 
\begin{rem}
\begin{enumerate}
    \item Since the long root of $B_n$ forms a root system of type $D_n$,
		 it follows that $\mathring{\Psi}\backslash (B_{J_e}\sqcup B_{J_o})$ is a union of root systems of type $A$ and $D$ (\cite{BdS}).
		 
		 \item The converse of \Cref{semiclosedD} is also true.
		 
\end{enumerate}

\end{rem}

 \begin{prop}\label{propdn+1}
	Let $\Psi$ be a symmetric closed subset of $\Phi$ such that $Gr(\Psi)$ is semi-closed. If \eqref{decomppsi} and \eqref{decompgrpsi} are the decomposition of $\Psi$ and $Gr(\Psi)$ respectively into irreducible components, then for each $i,$ there exists a $\bz$-linear function $p_i:Gr(\Psi_i)\to \bz,\ \alpha\mapsto p_i(\alpha)\in Z_\alpha,$ and 
	 $n_i\in 2\bz_+$ such that $$\Psi_i=\{\alpha+(p_i(\alpha)+n_i\bz)\delta:\alpha\in Gr(\Psi_i)\}.$$
	In particular, $\Psi$ is a closed subroot system of $\Phi.$ Moreover, we have $r\ge 2$, and $Gr(\Psi_1)$, $Gr(\Psi_2)$ are of type $B$, and the rest of $Gr(\Psi_i)$'s are of either type  $A$ or $D$, and $p_1(\alpha)\in 2\bz+1\,(\text{resp.}\ p_2(\alpha)\in 2\bz)$ for a short root $\alpha\in Gr(\Psi_1)\,(\text{resp.}\ Gr(\Psi_2)).$

	\end{prop}
	\begin{proof}
Since $Gr(\Psi)$ is semi-closed, by \Cref{semiclosedD}, we have $Gr(\Psi)= \mathring{\Psi}_1\sqcup \mathring{\Psi}_2\sqcup \cdots \sqcup \mathring{\Psi}_r$ where 
$\mathring{\Psi}_1$ and $\mathring{\Psi}_2$ are of type $B.$ For $1\le i\le r,$ let $\Psi_i:=\{\alpha+r\delta\in \Psi\mid \alpha\in \mathring{\Psi}_i\}.$ Then $\Psi=\Psi_1\sqcup \cdots \sqcup \Psi_r$ is the decomposition of $\Psi$ into irreducible components. It is easy to see that each $\Psi_i$ is a symmetric closed subset of $\Phi$ with $Gr(\Psi_i)$ is closed in $\mathring{\Phi}$. From the proof of \cref{semiclosedD}, we have $|\pi_m(Z_\alpha)|=1$ for all $\alpha\in Gr(\Psi).$ Now \cref{prop:equal1} gives the desired result for the first part of the proof. Again the second part is clear from \cref{semiclosedD}.
    \end{proof}
 \subsection{Twisted real affine root system of type \texorpdfstring{$A_{2n-1}^{(2)}$}{A{2n-1}(2)}}  In this case, the real affine roots are given by
		$$\Phi=\{\pm 2\epsilon_i+2r\delta\,\mid r\in\bz,\,1\le i\le n\}\cup \{\pm \epsilon_i\pm\ep_j+r\delta,\mid r\in\bz,\,1\le i\neq j\le n\}.$$
	The next proposition is the main result of this subsection.

	\begin{prop}\label{a22n-1}
		Let $\Psi$ be a symmetric closed subset of $\Phi$ such that $Gr(\Psi)$ is semi-closed in $\mathring{\Phi}$. 
		Let $\Psi=\Psi_1\sqcup \cdots \sqcup \Psi_r$ be the decomposition of $\Psi$ into irreducible components. 
		Assume that $Gr(\Psi_i)$ is not of type $A_1$ or $B_2$ for each $1\le i\le r.$
		Then we have $\Psi$ is a closed subroot system.

	\end{prop}
	
	\begin{proof}
		
		Define $I:=\{i\in I_n\mid 2\epsilon_i\in Gr(\Psi)\}$. Since $Gr(\Psi)$ is semi-closed, there exist $\epsilon_s-\epsilon_t,\epsilon_s+\epsilon_t\in Gr(\Psi)$ but $2\epsilon_s\notin Gr(\Psi)$. Hence $I\neq I_n$. Let 
		$$\Psi':=\left(\bigcup_{i\in I} (\pm2\epsilon_i+2\bz\delta) \cup \bigcup_{i\neq j \in I} (\pm\ep_i\pm\ep_j+\bz\delta) \right)\cap \Psi,\ \ \Psi'':=\Psi\backslash \Psi'.$$
	    If $\epsilon_k\pm \epsilon_\ell\in Gr(\Psi)$ for some $k\in I$, but $\ell\in I_n\backslash I$. Then we have 
	    $2\ep_k+2p\delta, \ep_k\pm\ep_\ell+r\delta \in \Psi$ for some $p, r\in \mathbb{Z}.$ This implies that
	    $$\ep_k\mp \ep_\ell+(2p-r)\delta=\,(2\ep_k+2p\delta)-(\ep_k\pm\ep_\ell+r\delta)\in\Psi,$$ and hence  $\mp2\ep_\ell+2(p-r)\delta=(\ep_k\mp \ep_\ell+(2p-r))-(\ep_k\pm\ep_\ell+r\delta)\in\Psi$, which is a contradiction as $\ell\notin I$. So if $\epsilon_i\pm \epsilon_j\in Gr(\Psi)$, then either both $i,j\in I$ or both $i,j\notin I.$ In particular, we have
	    $$\Psi''= \left( \bigcup_{i\neq j \in  I_n\backslash I} (\pm\ep_i\pm\ep_j+\bz\delta)\right)\cap \Psi$$ is a closed subroot system of $\Phi$,
	    and $\Psi=\Psi'\sqcup \Psi''$ is an orthogonal decomposition.  
	    
	    
	    \medskip
	    It is easy to see that $\Psi'$ is a symmetric closed subset of $\Phi$ with $Gr(\Psi')$ is a symmetric closed subset of $\mathring{\Phi}.$ 
	 Then by \cref{prop:notb2},  $\Psi'$ is also a closed subroot systems of $\Phi$, and hence $\Psi$ is a closed subroot system of $\Phi$. 
	\end{proof}
\begin{rem}
Let $\Psi_+=(2\epsilon_1+4\bz\delta)\cup (2\epsilon_2+4\bz\delta)\cup (\ep_1\pm\ep_2+ (4\bz\cup \pm 1+4\bz)\delta)\cup \{\ep_i+\ep_j+ (1+4\bz)\delta,\ep_i-\ep_j+ 4\bz\delta: 3\le i\neq j\le 5\}$
and set $\Psi=\Psi_+\cup -\Psi_+$. Then $\Psi$ is not a closed subroot system. So the condition that any irreducible component of $Gr(\Psi)$ is not of type $A_1$ or $B_2$ cannot be dropped from \cref{a22n-1}. 
\end{rem}	
	
\subsection{Twisted real affine root system of type \texorpdfstring{$D_{4}^{(3)}$}{D{4}(3)}}
	In this case, the real roots are given by
		$$\Phi=\{\pm \,(\epsilon_i-\epsilon_j)+r\delta,\,\pm\,(2\epsilon_i-\epsilon_j-\ep_k)+3r\delta\,\mid i,j,k\in I_3, i\neq j,\,r\in\bz\}$$
Note that $\mathring{\Phi}$ is of type $G_2$ and both $\mathring{\Phi}_s$ and  $\mathring{\Phi}_\ell$ are of type $A_2.$ 	The next lemma is a generalization of \cite[Theorem 2]{DV21} and \cite[Lemma 7.1.1]{RV19}.

	\begin{lem}\label{g2short}
		Let $\Psi$ be a symmetric closed subset of $\Phi$ with $Gr(\Psi)$ semi-closed. Then $Gr(\Psi)$ is $\mathring{\Phi}_s$.
	\end{lem}
	
	\begin{proof}
		Since $Gr(\Psi)$ is semi-closed, there exist short roots $\alpha,\beta\in Gr(\Psi)$ such that $\alpha+\beta\in \mathring{\Phi}_\ell$ and $\alpha+\beta\notin Gr(\Psi)$. 
		Now since $\alpha+\beta$ is long in type $G_2$, $\alpha-\beta$ must be short and hence $\alpha-\beta\in Gr(\Psi)$. Since $\alpha-\beta\not=\alpha, \beta$ and $Gr(\Psi)$ is symmetric, we have $\mathring{\Phi}_s\subseteq Gr(\Psi)$.
		Note that $Gr(\Psi)$ can not contain two long positive roots as $Gr(\Psi)\subsetneq \mathring{\Phi}$. We claim that $Gr(\Psi)$ can not contain any long root. On the contrary assume that $Gr(\Psi)$ contains a long root, say $2\epsilon_i-\epsilon_j-\ep_k\in Gr(\Psi)$. Let $S=\pi_m(Z_{\epsilon_j-\epsilon_k}).$ It is clear that $|S|\neq 3.$ Suppose that $|S|=2$. Let $\epsilon_i-\epsilon_j+p_1\delta,\epsilon_j-\ep_k+p_2\delta,\epsilon_i-\ep_k+p_3\delta\in \Psi$. Since $Gr(\Psi)=\mathring{\Phi}_s\cup \{\pm (2\epsilon_i-\epsilon_j-\ep_k)\}$ we have 
		\begin{equation}\label{rel1}
		    p_1\not\equiv p_2\ (\mathrm{mod}\ 3), \ \ p_3\not\equiv -p_2\ (\mathrm{mod}\ 3)
		\end{equation}
		Let $\ell$ be the element which is not in $S$. Then $p_1\equiv \ell(\mathrm{mod}\ 3)$ and $p_3\equiv -\ell(\mathrm{mod}\ 3)$ which gives $|\pi_m(Z_{\epsilon_i-\epsilon_k})|=1,$ but this is impossible since we have $Z_{\epsilon_i-\epsilon_j}+Z_{\epsilon_j-\ep_k}\subseteq Z_{\epsilon_i-\ep_k}.$
		
		\medskip
		Now assume that $|S|=1$ and let $S=\{\ell'\}$. In this case,  we have one more relation along with \eqref{rel1} namely 
		\begin{equation}\label{rel2}
			p_3-p_1\equiv \ell'\ (\mathrm{mod}\ 3)
		\end{equation}
		which implies that $|\pi_m(Z_{\epsilon_i-\epsilon_j})|=1$ and $|\pi_m(Z_{\epsilon_i-\ep_k})|=1.$ Suppose that $\pi_m(Z_{\epsilon_i-\epsilon_j})=\{r\}$. Then we have $\pi_m(Z_{\epsilon_i-\ep_k})=\{r+\ell'\}$. Since $\Psi$ is closed, for any $2\epsilon_i-\epsilon_j-\ep_k+3a\delta\in\Psi,$ we have $$\epsilon_i-\ep_k+(3a-p_1)\delta\,=\,(2\epsilon_i-\epsilon_j-\ep_k+3a\delta)+(\epsilon_j-\epsilon_i-p_1\delta)\in\Psi.$$ This implies $3a-p_1\equiv r+\ell'\ (\mathrm{mod}\ 3),$ which in turn implies that $\ell'\equiv -2r\equiv r\ (\mathrm{mod}\ 3).$ This contradicts \cref{rel1}. So $|S|\neq 1$ and hence the result.
	\end{proof}
	
\begin{prop}\label{prop: D43}
	Let $\Psi$ be a closed symmetric subset of $\Phi$ with $Gr(\Psi)$ semi-closed in $\mathring{\Phi}$. Then $\Psi$ is of the form 
	$$\Psi=\pm (\epsilon_i-\epsilon_j+r\bz\delta)\cup\,\pm(\epsilon_j-\ep_k+(r\bz+\ell)\delta)\cup\,\pm (\epsilon_i-\ep_k+(r\bz+\ell)\delta)$$ where $r\in 3\bz_+$ and $\{i,j,k\}$ is a permutation of $\{1,2,3\}$. In particular, $\Psi$ is a closed subroot system of $\Phi$.
\end{prop}
	
\begin{proof}
	By \cref{g2short}, it follows that $Gr(\Psi)=\mathring{\Phi}_s$. Note that we can consider $\Psi$ as an irreducible symmetric closed subset of the untwisted root system of type $A_2^{(1)}$. By \cref{coruntwisted}, there is a $\bz$-linear function $p:Gr(\Psi)\to \bz, \alpha\mapsto p_\alpha\in Z_\alpha$ and $r\in \bz_+$ such that $$\Psi=\{\alpha+(p_\alpha+r\bz)\delta:\alpha\in Gr(\Psi)\}$$ In particular, $\Psi$ is a closed subroot system of $\Phi.$ Next we determine $r$ and $p_\alpha$ explicitly. 
	At first, we claim that $0\in \pi_m(Z_\alpha)$ for some $\alpha\in Gr(\Psi)\cap \mathring{\Phi}^+$. Suppose that $\ep_1-\ep_2+r_1\delta,\,\ep_2-\ep_3+r_2\delta,\,\ep_1-\ep_3+r_3\delta\in \Psi$. We then have the following relations:
	\begin{equation}\label{d34}
	    	r_1\not\equiv r_2\ (\mathrm{mod}\ 3),\ \ -r_3\not\equiv r_2\ (\mathrm{mod}\ 3),\ \ -r_3\not\equiv r_1\ (\mathrm{mod}\ 3).
	\end{equation}
    This implies that the integers $r_1,r_2,-r_3\ (\mathrm{mod}\ 3)$  are all distinct, so one of them must be $0\ (\mathrm{mod}\ 3)$.
    With out loss of generality we assume that $r_1\equiv 0\ (\mathrm{mod}\ 3).$ By similar argument we get
     $r_1',r_2,-r_3\ (\mathrm{mod}\ 3)$  are all distinct for any $r'\in Z_{\epsilon_1-\epsilon_2}$ and using \ref{d34} we get
     $r_1'\equiv 0\ (\mathrm{mod}\ 3)$ and $r_2\equiv r_3\ (\mathrm{mod}\ 3).$ This implies 
     $\Psi$ is of the form 
	$$\Psi=\pm (\epsilon_i-\epsilon_j+r\bz\delta)\cup\,\pm(\epsilon_j-\ep_k+(r\bz+\ell)\delta)\cup\,\pm (\epsilon_i-\ep_k+(r\bz+\ell)\delta)$$ where $r\in3\bz_+$ and $\{i,j,k\}$ is a permutation of $\{1,2,3\}$. 
	\end{proof}

\subsection{Twisted real affine root system of type \texorpdfstring{$E_6^{(2)}$}{E{6}(2)}}
In this case, the real roots
are given by $$\Phi=\big\{\pm\epsilon_i+r\delta, \pm\epsilon_i\pm\epsilon_j+2s\delta : 1\le i\neq j\le 4, r, s\in \bz\big\}\cup\{\sum_{i=1}^{4} \pm\epsilon_i/2+r\delta : r\in \bz\}.$$
The gradient root system is of type $F_4,$ and the short roots $\mathring{\Phi}_s$ form a root system of type $D_4$. 
For convenience, we call a short root of the form $\frac{1}{2}\left(\sum_{i=1}^4\pm \ep_{i}\right)$ as a \textit{special short root}.
For $I\neq\emptyset$ and a special short root $\mu=\frac{1}{2}\,(\sum_{i=1}^4\nu_{i}\ep_{i})$ define
$D_\mu^+(\Psi,I):=\{\epsilon_i\mid i\in I\}\cup\{\sum_{i=1}^4 \lambda_i\epsilon_i/2\,\mid\, \lambda_i\in \{\pm 1\}, i\in I, \lambda_j=\nu_j\text{ for } j\notin I\}$
and set $$D_\mu(\Psi,I)=D_\mu^+(\Psi,I)\cup -D_\mu^+(\Psi,I).$$
Let $\Psi$ be a symmetric closed subset of $\Phi$ such that Gr$(\Psi)$ is semi-closed.  If $Gr(\Psi)$ does not contain any special short root, then $\Psi$ can be realized as a symmetric, closed subset of the real affine root system of type $D_5^{(2)}$ and hence it is a closed subroot system (see \cref{propdn+1}).  So without loss of generality, we can assume that there is a special short root in $Gr(\Psi)$, say $\alpha=\frac{1}{2}\,(\sum_{i=1}^4\nu_{i}\ep_{i})\in Gr(\Psi).$  
Let $I=\{i\in I_4\mid \epsilon_i\in Gr(\Psi)\}.$ Since $Gr(\Psi)$ is semi-closed, there exist short roots $\beta,\gamma\in Gr(\Psi)$ such that $\beta+\gamma$ is a long root and $\beta+\gamma\notin Gr(\Psi).$ We claim that we can always choose $\beta,\gamma$ to be special short roots. For $|I|\le 1,$ it is obvious. Suppose $|I|\geq 2,$ and there are
$\mu_k\ep_{k}, \mu_r\ep_{r}\in Gr(\Psi)$ with $k\neq r$, but $\mu_k\ep_{k}+\mu_r\ep_{r}\notin Gr(\Psi)$. Then since $\alpha=\frac{1}{2}\,(\sum_{i=1}^4\nu_i\epsilon_i)\in Gr(\Psi)$ and $Gr(\Psi)$ is semi-closed, we have that $\beta=\frac{1}{2}\left(\mu_k\epsilon_k+\mu_r\epsilon_r+\sum\limits_{i\notin \{k, r\}}\nu_i\epsilon_i\right)$ and
$\gamma=\frac{1}{2}\left(\mu_k\epsilon_k+\mu_r\epsilon_r+\sum\limits_{i\notin \{k, r\}}-\nu_i\epsilon_i\right)$ are in $Gr(\Psi).$ Clearly
$\beta+\gamma=\mu_k\epsilon_k+\mu_r\epsilon_r.$

\begin{prop}\label{prop:10}
Let $\Psi$ be a symmetric, closed subset of $\Phi$ such that no irreducible component of Gr$(\Psi)$ is of type $A_1$ and $Gr(\Psi)$ is semi-closed. Set $I=\{i\in I_4:\epsilon_i\in Gr(\Psi)\}$, then  we have $|I|\equiv 0\ (\mathrm{mod}\ 2$) and exactly one of the following holds.
\begin{enumerate}
    \item If $|I|=2,$ then $\mathring{\Phi}_s\supseteq Gr(\Psi)=D_\beta(\Psi,I)$ where $\beta$ is a special short root in $Gr(\Psi).$ There exists a $\bz$-linear function $p:D_\beta(\Psi,I)\to \bz,\  p_\alpha\in Z_\alpha,\alpha\in D_\beta(\Psi,I)$ and $n\in 2\bz_{+}$ such that 
    $$\Psi=\{\alpha+(p_\alpha+n\bz)\delta\mid \alpha\in D_\beta(\Psi,I)\}.$$

    \item If $|I|=4,$ then there exists a partition $J_1,J_2$ of $I_4$ such that $|J_1|=|J_2|=2$ and a $\bz$--linear function $p:Gr(\Psi)\to \bz$ with $p_{\ep_{i}}\in 2\bz$ 
    (resp.  $p_{\ep_{i}}\in 2\bz+1$) if $i\in J_1$ (resp. $i\in J_2$) such that
	$$\Psi:=\bigcup_{\alpha\in \mathring{\Phi}_s}\{\alpha+(p_\alpha+n\bz)\delta\}\cup \bigcup_{J\in \{J_1,J_2\}}\{\pm(\ep_i\pm \ep_j)+(p_{\pm(\ep_i\pm \ep_j)}+n\bz)\delta:i\neq j\in J\}.$$

    \item If $I=\emptyset,$ there exists a $\bz$-linear function $p:Gr(\Psi)\to \bz$ and $m_0,m_1\in 2\bz_{+}$ such that $\Psi=\Psi_0\cup \Psi_1$ where $\Psi_i=\{\alpha+(p_\alpha+m_i\bz)\delta\mid \alpha\in Gr(\Psi_i)\}.$  Moreover, $Gr(\Psi)$ is a root system of type $B_2\times B_2$ and we have $p_\alpha\in 2\bz\ (resp.\ 2\bz+1)$ if $\alpha\in Gr(\Psi_0)$ (resp. $Gr(\Psi_1)$ and $\alpha$ is short).
\end{enumerate}
\end{prop}
\begin{proof}
Since $Gr(\Psi)$ is semi-closed, there exist special short roots $\beta,\gamma\in Gr(\Psi)$ such that the $\beta+\gamma$ is a long root and $\beta+\gamma\notin Gr(\Psi)$. Let $\beta=\frac{1}{2}\,(\lambda_{s_1}\ep_{s_1}+\lambda_{s_2}\ep_{s_2}+\lambda_{s_3}\ep_{s_3}+\lambda_{s_4}\ep_{s_4})$ and $\gamma=\frac{1}{2}\,(\lambda_{s_1}\ep_{s_1}+\lambda_{s_2}\ep_{s_2}-\lambda_{s_3}\ep_{s_3}-\lambda_{s_4}\ep_{s_4})$. Note that we have $(Z_\beta+Z_\gamma)\cap 2\bz=\emptyset$, so we have $\beta-\gamma\notin Gr(\Psi)$ and $|\pi_m(Z_\beta)|=1$.
First we assume that $I\neq \emptyset$. We claim that  $\ep_{s_1}\,($resp. $\ep_{s_3})\in Gr(\Psi)$ if and only if $\ep_{s_2}\,($resp. $\ep_{s_4})\in Gr(\Psi)$.
Suppose $s_1\in I$, then we have $\gamma'=\gamma-\lambda_{s_1}\ep_{s_1}\in Gr(\Psi).$ This implies that $\lambda_{s_2}\ep_{s_2}=\gamma'+\beta\in Gr(\Psi)$ and the case
$s_3\in I$ is done similarly.
This argument also shows that for any $i\in I$ and a special short root 
$\gamma\in Gr(\Psi)$, we can always change the sign of the co-efficient of $\ep_i$ in $\gamma$ and that resulted element is again in $Gr(\Psi).$

It is easy to see that we always have $D_\beta(\Psi,I)\subseteq Gr(\Psi)\cap \mathring{\Phi}_s.$  
We claim that $Gr(\Psi)\cap \mathring{\Phi}_s=D_\beta(\Psi,I).$ Suppose $|I|=4,$ then it is clear as $D_\beta(\Psi,I)=\mathring{\Phi}_s$. So assume that $|I|=2.$ 
If possible let $Gr(\Psi)\cap \mathring{\Phi}_s\backslash D_\beta(\Psi,I)\neq \emptyset$, then
there exists $\mu\in Gr(\Psi)\cap \mathring{\Phi}_s\backslash D_\beta(\Psi,I)$ and unique
$j\notin I$ such that the co-efficient of $\ep_j$ in both $\mu$ and $\beta$ is different. 
But this implies $\beta-\mu=\pm \ep_j\in  Gr(\Psi)\cap \mathring{\Phi}_s$, which is a contradiction.
Thus we get $$Gr(\Psi)\cap \mathring{\Phi}_s = D_\beta(\Psi,I).$$
Also since sum of a long root and a short root is short, it follows that $\epsilon_s\pm\epsilon_t\in Gr(\Psi)$ only if $\{s,t\}\subseteq I$ or $\{s,t\}\subseteq I_4\backslash I$. Note that every special short root $\mu\in D_\beta(\Psi,I)$ can be written as $\mu=\beta+\sum \pm \epsilon_i$ with $\epsilon_i\in Gr(\Psi)$ such that each partial sum is a root. Since all the roots appearing are short and $|\pi_m(Z_\beta)|=1,$ it follows that $|\pi_m(Z_\mu)|=1$ for each $\mu\in D_\beta(\Psi,I).$
So we have $|\pi_m(Z_\alpha)|=1$ for all $\alpha\in Gr(\Psi).$

\medskip
For $|I|=2,$ we show that no long root can occur in $Gr(\Psi)$. We shall show the case for $I=\{s_1,s_2\}$ and the remaining cases are similar. Since $\lambda_{s_1}\ep_{s_1}+\lambda_{s_2}\epsilon_{s_2}\notin Gr(\Psi)$, it follows that $(Z_{\ep_{s_1}}+Z_{\epsilon_{s_2}})\cap 2\bz=\emptyset$. Hence $\lambda_{s_1}\epsilon_{s_1}-\lambda_{s_2}\epsilon_{s_2}\notin Gr(\Psi)$ as well. Recall that $\epsilon_s\pm\epsilon_t\in Gr(\Psi)$ only if $\{s,t\}\subseteq I$ or $\{s,t\}\subseteq I_4\backslash I$. So we only can have  $\lambda_{s_3}\epsilon_{s_3}\pm \lambda_{s_4}\epsilon_{s_4}\in Gr(\Psi)$.
Since $Gr(\Psi)$ has no component of type $A_1$ and $\lambda_{s_3}\epsilon_{s_3}-\lambda_{s_4}\epsilon_{s_4}$ is orthogonal to $D_\beta(\Psi,I)$ and 
$\lambda_{s_3}\epsilon_{s_3}+\lambda_{s_4}\epsilon_{s_4}$, we have
 $$\lambda_{s_3}\epsilon_{s_3}-\lambda_{s_4}\epsilon_{s_4}\notin Gr(\Psi).$$
 Suppose  $\lambda_{s_3}\epsilon_{s_3}+ \lambda_{s_4}\epsilon_{s_4}\in Gr(\Psi)$. We have $(Z_\beta+Z_\gamma)\cap2\bz\not=\emptyset$
 since $\beta=\gamma+\lambda_{s_3}\epsilon_{s_3}+\lambda_{s_4}\epsilon_{s_4} \in Gr(\Psi)$, which in turn implies that $\beta+\gamma\in Gr(\Psi)$, a contradiction.
 Hence no long root can occur in $Gr(\Psi)$ if $|I|=2.$
  In this case we have $Gr(\Psi)=D_\beta(\Psi,I)$ which is irreducible.
When $|I|=4$ we have $D_\beta(\Psi,I)=\mathring{\Phi}_s\subseteq \Psi$. So we must have $Gr(\Psi)$ irreducible in this case.   
Now using \cref{prop:1}, we get a $\bz$-linear function $p$ and $n\in 2\mathbb{Z}_+$ such that
$\Psi=\{\alpha+(p_\alpha+n\bz)\delta:\alpha\in Gr(\Psi)\}.$ This completes the proof for $I\neq \emptyset$.
Moreover when $I=I_4$ we have the following restrictions:
	$$|\pi_m(Z_{\ep_{s_i}})|=1,\ \ \forall i \in I_4,\ \ (Z_{\ep_{s_1}}+Z_{\ep_{s_2}})\cap 2\bz=\emptyset, \ \ (Z_{\ep_{s_3}}+Z_{\ep_{s_4}})\cap 2\bz=\emptyset.$$
	Set $J_1:=\{s_i\in I_4:Z_{\ep_{s_i}}\subseteq 2\bz\}$ and $J_2:=\{s_i\in I_4:Z_{\ep_{s_i}}\subseteq 2\bz+1\}.$ From the relations above, we have $|J_1|=|J_2|=2.$ Then there exists a $\bz$--linear function $p:Gr(\Psi)\to \bz$ such that $p_{\ep_{s_i}}\in 2\bz$ (resp. in $2\bz+1$) if $s_i\in J_1$ (resp. in $J_2$) and $\Psi$ is of the form 
	$$\Psi:=\bigcup_{\alpha\in \mathring{\Phi}_s}\{\alpha+(p_\alpha+n\bz)\delta\}\cup \bigcup_{J\in \{J_1,J_2\}}\{\pm(\ep_i\pm \ep_j)+(p_{\pm(\ep_i\pm \ep_j)}+n\bz)\delta:i\neq j\in J\}.$$

\medskip
Assume that $I=\emptyset$. It is clear that the only short roots can appear in $Gr(\Psi)$ are $\alpha_1:=\beta,\alpha_2:=\gamma,\alpha_3=\frac{1}{2}\,(\lambda_{s_1}\epsilon_{s_1}-\lambda_{s_2}\epsilon_{s_2}+\lambda_{s_3}\ep_{s_3}-\lambda_{s_4}\epsilon_{s_4}),\alpha_4=\frac{1}{2}\,(-\lambda_{s_1}\epsilon_{s_1}+\lambda_{s_2}\epsilon_{s_2}+\lambda_{s_3}\epsilon_{s_3}-\lambda_{s_4}\epsilon_{s_4})$ and their negatives (see \cite[Proposition 8.1.3]{RV19}). We claim that all of these roots occur in $Gr(\Psi)$. As $(\alpha_i,\alpha_j)=0,\ 1\le i\neq j \le 4,$ $Gr(\Psi)\not=\{\pm\alpha_i\mid i=1,2\}.$ Hence there is a long root in $Gr(\Psi)$. Since $\alpha_1\pm\alpha_2\notin Gr(\Psi)$, there is a long root in Gr$(\Psi)$ of the form $\lambda_{s_i}\epsilon_{s_i}\pm \lambda_{s_j}\epsilon_{s_j}$ with $i\in\{1,2\}$ and $j\in\{3,4\}$. We shall prove the claim for one case, other cases are similar. Let $\lambda_{s_1}\epsilon_{s_1}+ \lambda_{s_3}\epsilon_{s_3}\in Gr(\Psi)$. Then $-\alpha_3=\alpha_1-(\lambda_{s_1}\epsilon_{s_1}+ \lambda_{s_3}\epsilon_{s_3})\in Gr(\Psi)$. Furthermore since $Z_{\alpha_1}$ contains elements of same parity, it follows that $(Z_{\alpha_1}+Z_{\alpha_3})\cap (2\bz+1)=\emptyset.$ Hence $\alpha_1+\alpha_3=\lambda_{s_2}\ep_{s_2}+ \lambda_{s_4}\epsilon_{s_4}\in Gr(\Psi)$. Since $Gr(\Psi)$ has no irreducible component of type $A_1$ and $\alpha_2$ is orthogonal to $\alpha_1,\alpha_3,\lambda_{s_1}\ep_{s_1}+ \lambda_{s_3}\ep_{s_3}$, it follows that $\lambda_{s_2}\ep_{s_2}- \lambda_{s_4}\ep_{s_4}\in Gr(\Psi)$. Consequently $-\alpha_4=\alpha_2-(\lambda_{s_2}\ep_{s_2}- \lambda_{s_4}\ep_{s_4})\in Gr(\Psi)$. By a similar argument we get that $\lambda_{s_1}\ep_{s_1}- \lambda_{s_3}\ep_{s_3}\in Gr(\Psi)$.
Hence $|\pi_m(Z_\gamma)|=1$ for all $\gamma\in Gr(\Psi).$ Moreover there exist a partition $J_0\sqcup J_1$ of $I_4$ with $|J_i|=2,\ i=0,1$ such that $Gr(\Psi)=\mathring{\Psi}_0\cup \mathring{\Psi}_1$ where $\mathring{\Psi}_i=\{\pm \alpha_k,\pm \alpha_\ell,\pm(\alpha_k\pm \alpha_\ell):J_i=\{k,\ell\}\}.$ We also have $\alpha_k\pm\alpha_\ell\notin Gr(\Psi)$ if $k\in J_i$ but $\ell\notin J_i.$ We can choose $\mathring{\Psi}_0$ to be the component so that $Z_\gamma\subseteq 2\bz$ for all $\gamma\in \mathring{\Psi}_0.$ In each component, all the hypothesis of \cref{prop:1} is satisfied. Hence $\Psi$ is of the required form
with exactly two of $p_{\alpha_i}$'s being even.

\end{proof}

\section{Twisted real affine root system with non-reduced gradient}
    In this section, we shall consider the twisted affine root system of type $A_{2n}^{(2)}.$ 
    The real affine roots are $$\Phi=\left\{\pm\epsilon_i+(r+1/2)\delta,\,\pm 2\epsilon_i+2r\delta,\,\pm \epsilon_i\pm\epsilon_j+r\delta\mid 1\le i\neq j\le n,r\in\bz \right\}.$$
    The gradient root system is an irreducible reduced root system of type $BC_n$. We set
    $$\mathring{\Phi}_d := \{\pm2\epsilon_i\mid 1\le i\le n\}, \ \mathring{\Phi}_s := \{\pm \epsilon_i\mid 1\le i\le n\}, \ \mathring{\Phi}_\ell := \{\pm\epsilon_i\pm\epsilon_j\mid 1\le i\neq j\le n\}$$ and 
    $$\widehat{C_n}=\left\{\pm 2\epsilon_i+2r\delta,\,\pm \epsilon_i\pm\epsilon_j+r\delta\mid 1\le i\neq j\le n,r\in\bz \right\}\subseteq \Phi.$$
    Note that $\widehat{C_n}$ is a closed subroot system of type $A_{2n-1}^{(2)}$. Let $\Psi\subseteq \widehat{C_n}$, then we have $\Psi$ is a symmetric closed in $\widehat{C_n}$ if and only if
    $\Psi$ is a symmetric closed in $\Phi.$ We will fix some notations before proceeding further.
    For any $I\subseteq I_n$, set $$B_I=\{\pm \epsilon_i, \pm \epsilon_i\pm\epsilon_j \mid i\neq j\in I\},\, \, C_I=\{\pm 2\epsilon_i, \pm \epsilon_i\pm\epsilon_j \mid i\neq j\in I\}$$ and
    $BC_I=\{\pm \epsilon_i, \pm 2\epsilon_i, \pm \epsilon_i\pm\epsilon_j \mid i\neq j\in I\}$.
    
    \subsection{}
	Let $\Psi$ be a symmetric closed subset of $\Phi$. The next lemma shows that it is enough to classify irreducible symmetric, closed subsets of $\Phi$ which contains at least one short root.
	\begin{lem}\label{irrpsi}
	    Let $\Psi$ be a symmetric closed subset of $\Phi$ and \cref{decomppsi} be the decomposition of $\Psi$ into irreducible components. Then there exists at most one $\Psi_i, 1\le i\le r$ such that $Gr(\Psi_i)\cap \mathring{\Phi}_s\neq \emptyset.$  
	    Conversely, suppose that $\Psi'$ is an irreducible symmetric, closed subset of $\Phi$ such that $Gr(\Psi')\cap \mathring{\Phi}_s\neq \emptyset$ and $\Psi''$ is a symmetric, closed subset of $\wh{C_n}$ such that $(\Psi',\Psi'')=0.$ Then $\Psi'\cup \Psi''$ is a symmetric, closed subset of $\Phi.$
	\end{lem}
	\begin{proof}
	If possible, assume that there exists $i,j$ with $1\le i\neq j\le r$ such that $\epsilon_k\in Gr(\Psi_i), \epsilon_\ell\in Gr(\Psi_j).$ Then we have $\epsilon_k+q\delta\in \Psi, \epsilon_\ell+s\delta\in \Psi$ for some $q,s\in \bz+1/2.$ Since $\Psi$ is closed, we have $\epsilon_k+\epsilon_\ell+(q+s)\delta\in \Psi.$ But the root $\epsilon_k+\epsilon_\ell+(q+s)\delta$ is non-orthogonal to both $\Psi_i$ and $\Psi_j$, which is impossible. For the converse part, it is easy to see that $\Psi'\cup \Psi''$ is a symmetric subset of $\Phi.$ Now we claim if $\alpha\in \Psi'$ and $\beta\in \Psi''$ then $\alpha+\beta\notin \Phi.$ This clearly implies $\Psi'\cup \Psi''$ is closed in $\Phi$ as both $\Psi'$, $\Psi''$ are closed in $\Phi.$
	Set $$I'=\{i\in I_n\mid \epsilon_i\in Gr(\Psi')\}.$$ Since $\Psi$ is closed, we have $\ep_i\pm \ep_j\in Gr(\Psi')$ for all $i\neq j\in I'$ and $\ep_i\pm \ep_j\notin Gr(\Psi')$ for all $i\in I',j\in I_n\backslash I'.$ Moreover, since $\Psi'$ is irreducible, we have $B_{I'}\subseteq Gr(\Psi')\subseteq BC_{I'}.$
	Also we have that $Gr(\Psi'')\subseteq \mathring{\Phi}_\ell\cup \mathring{\Phi}_d$ and it is orthogonal to $Gr(\Psi').$ So it follows that 
	$Gr(\Psi'')\subseteq C_{I_n\backslash I'}.$	Hence if $\alpha\in Gr(\Psi'),\beta\in Gr(\Psi''),$ then $\alpha+\beta\notin \mathring{\Phi}$. This gives our claim and completes the proof.
	
	\end{proof}

\subsection{}Now onwards, we assume that $\Psi$ is an irreducible symmetric, closed subset of $\Phi$ such that $Gr(\Psi)\cap \mathring{\Phi}_s\neq \emptyset.$ Let 
$I:=\{i\in I_n\mid \epsilon_i\in$ Gr$(\Psi)\}$, and we have $I\neq \emptyset.$ 
	Let $q_i\in Z_{\epsilon_i}$ for each $i\in I$ and set $\tilde{Z}_{\epsilon_i}:=Z_{\epsilon_i}-q_i$. Then we have $\tilde{Z}_{\epsilon_i}\subseteq \bz$ for $i\in I.$ 
	The next two lemmas are important and we give proof for the first one and skip the details of the second one as it is similar to the first one. Recall that $m=2.$
	
	\begin{lem}\label{a2nequal1}
		Let $\Psi$ be an irreducible symmetric closed subset of $\Phi$  and $I$ be defined as above. Suppose that $|\pi_m(\tilde{Z}_{\ep_i})|=1$ for some $i\in I.$ Then
		\begin{enumerate}
			\item $|\pi_m(\tilde{Z}_{\ep_j})|=1$ for all $j\in I$.
			\item $Gr(\Psi)=B_I$ and $|\pi_m(Z_{\epsilon_i\pm \epsilon_j})|=1.$
		\end{enumerate}
	\end{lem}
	
	\begin{proof}
	Let $j\in I$ such that $j\neq i$. Since $\Psi$ is closed, we have $\pm(\epsilon_i\pm \epsilon_j)\in Gr(\Psi)$ and $Z_{\epsilon_i\pm\epsilon_j}+Z_{\mp\epsilon_j}= Z_{\epsilon_i}.$ Hence we have $\tilde{Z}_{\epsilon_i\pm\epsilon_j}+\tilde{Z}_{\mp\epsilon_j}= \tilde{Z}_{\epsilon_i}$,  $\tilde{Z}_{\epsilon_i\pm\epsilon_j}\subseteq \tilde{Z}_{\epsilon_i}$ and $\tilde{Z}_{\mp\epsilon_j}\subseteq \tilde{Z}_{\epsilon_i}$. This implies $(1)$ and second part of $(2)$. 
	Since $\Psi$ is closed,	we also get $\epsilon_r\pm \epsilon_s\notin Gr(\Psi)$ if $r\in I$ and $s\notin I.$ As $Gr(\Psi)$ is irreducible, we must have
	$B_I\subseteq  Gr(\Psi')\subseteq BC_I$.	If possible assume that $2\epsilon_i\in Gr(\Psi)$ for some $i\in I.$ Since $\Psi$ is closed, we have $Z_{2\epsilon_i}+Z_{-\epsilon_i}\subseteq Z_{\epsilon_i}.$ 
	For $s\in \tilde{Z_{i}}$ and $2k\in Z_{2\epsilon_i},$ we have $2k-s-q_i=2k-2q_i-s+q_i\in Z_{\ep_i}.$ This implies both $s$ and $2k-2q_i-s\in \tilde{Z}_{\ep_j}$
	which is a contradiction as they have different parities as $q_i\in \bz+1/2$. Thus we must have $Gr(\Psi)=\{\pm \epsilon_i,\pm(\epsilon_i\pm \epsilon_j):i\neq j\in I\}$.
	\end{proof}
	
		\begin{lem}\label{a2nneq1}
		Let $\Psi$ be an irreducible symmetric closed subset of $\Phi$ and $I$ be defined as above. Suppose that $|\pi_m(\tilde{Z}_{\ep_i})|=2$ for some $i\in I.$ Then
		\begin{enumerate}
	    	\item $|\pi_m(\tilde{Z}_{\ep_j})|=2$ for all $j\in I$.
			\item $Gr(\Psi)=BC_I$ and $|\pi_m(Z_{\epsilon_i\pm \epsilon_j})|=2.$
		\end{enumerate}\qed
	\end{lem}
\subsection{}	
    Generalizing \cite[Section 9.1]{RV19}, we now define certain symmetric, closed subsets of $\Phi$. For 
    $I\subseteq I_n$, $J\subseteq I, \tau\in 2\bz_{+}$, and a $\bz$-linear function
    $p:B_I\to \bz+1/2$ such that $p_{\epsilon_i}-\frac{1}{2}\in 2\bz\ (\text{resp.}\in 2\bz+1)$ for $i\in J\ (\text{resp.}\in I\backslash J),$ define
    $$\Psi_\tau^+(p,I,J):=\bigcup\limits_{\alpha\in B_I}(\alpha + (p_\alpha +\tau \bz)\delta),$$
    and $\Psi_\tau(p,I,J):=\Psi_\tau^+(p,I,J)\cup (-\Psi_\tau^+(p,I,J)).$ We first consider the situation that appears in \cref{a2nequal1}. In this case, we have

	\begin{prop}\label{a2n21}
	    Let $\Psi$ be an irreducible symmetric closed subset of $\Phi$ such that $Gr(\Psi)\cap \mathring{\Phi}_s\neq \emptyset.$ 
	    Let $I=\{i \in I_n : \epsilon_i\in Gr(\Psi)\}$ and assume that $|\pi_m(\tilde{Z}_{\ep_i})|=1$ for some $i\in I.$ Then
	    there exist $J\subseteq I\subseteq I_n$, $\tau\in 2\bz_{+}$, and a $\bz$-linear function $p:B_I\to \bz+1/2$ satisfying $p_{\epsilon_i}\in 2\bz+\frac{1}{2}\ (\text{resp.}\, p_{\epsilon_i}\in 2\bz+\frac{3}{2})$ for $i\in J\ (\text{resp.}\, i\in I\backslash J),$ such that $$\Psi=\Psi_\tau(p,I,J).$$
	    In particular, $\Psi$ is a closed subroot system of $\Phi.$
	\end{prop}
	
	\begin{proof}
	    We have that $Gr(\Psi)=B_I$ by \cref{a2nequal1}$(2).$ It is easy to check that $Z_\alpha+Z_\beta\subseteq Z_{\alpha+\beta}$ holds for all $\alpha,\beta,\alpha+\beta\in Gr(\Psi).$ Define $J:=\{j\in I\mid Z_{\ep_j}\subseteq 2\bz+\frac{1}{2}\}.$
	    Define a $\bz$-linear function $p:Gr(\Psi)\to \bz+1/2$ by choosing $p_i\in Z_{\ep_i},\ i\in I$ and extending $\bz$-linearly. 
	    Then we have that $p_{\epsilon_i}\in 2\bz+\frac{1}{2}\ (\text{resp.}\, p_{\epsilon_i}\in 2\bz+\frac{3}{2})$ for $i\in J\ (\text{resp.}\, i\in I\backslash J).$ As before, set $Z_\alpha':=Z_\alpha-p_\alpha, \ \alpha\in Gr(\Psi).$ Note that the argument of \cref{prop:1} goes through in this case, so we get that $A=Z_\alpha'=Z_\beta' \ \forall\, \alpha,\beta\in Gr(\Psi)$ and $A$ is a subgroup of $\bz,$ say $A=\tau\bz.$ Since $|\pi_m(Z_{\epsilon_i\pm \epsilon_j})|=1,$ we must have $\tau\in 2\bz_{+}.$ Thus $\Psi=\Psi_\tau(p,I,J).$
	\end{proof}
	
	\subsection{} For $k\in \bz_{+}, I\subseteq I_n$ and a $\bz$-linear function $p:BC_I\to \bz+1/2,$ define $\Psi_k(p,I):=\Psi^+_k(p,I)\cup (-\Psi^+_k(p,I))$ where $\Psi^+_k(p,I)$ is given by
	$$\Psi^+_k(p,I):=\bigcup\limits_{\alpha\in B_I}(\alpha+(p_\alpha+k\bz)\delta)\cup \bigcup\limits_{i\in I}(2\epsilon_i+(2p_{\epsilon_i}+k(2\bz+1))\delta).$$ 	We now consider the case that appears in \cref{a2nneq1}, i.e., when some $Z_{\ep_i}$ contains elements of different parities. In this case, we have
	\begin{prop}\label{a2n22}
		Let $\Psi$ be an irreducible symmetric closed subset of $\Phi.$ Let $I=\{i \in I_n : \epsilon_i\in Gr(\Psi)\}$ and assume that $|\pi_m(\tilde{Z_i})|=2$ for some $i\in I.$ 
		Then there is an odd integer $k\in\bz_+$ and a $\bz$-linear function $p:Gr(\Psi)\to \bz+1/2$ such that $\Psi=\Psi_k(p,I).$ In particular, $\Psi$ is a closed subroot system of $\Phi$.
	\end{prop}
	
	\begin{proof}
	    By \cref{a2nequal1}, we have $Gr(\Psi)=BC_I.$ It is easy to check that $Z_\alpha+Z_\beta\subseteq Z_{\alpha+\beta}$ whenever $\alpha\neq\beta\in \{\pm \epsilon_i:i\in I\}.$ 
	    We define $p:B_I\to\bz+1/2$ by choosing $p_i\in Z_{\epsilon_i}, \ i\in I$ and extending $\mathbb{Z}$-linearly. As before,
	    $Z_{\alpha}'=Z_\alpha-p_\alpha$ are equal and subgroups of $\bz$ for all $\alpha\in B_I.$
	    Since $|\pi_m(Z_{\epsilon_i\pm \epsilon_j})|=2$ holds by \cref{a2nequal1}, so we must have that $k$ is an odd integer. Write $Z_{\epsilon_i}=A\cup B$ where $A=p_{\epsilon_i}+2k\mathbb{Z},\ B=p_{\epsilon_i}+k(1+2\mathbb{Z})$. Since $A+B\subseteq 2\bz$ and $\Psi$ is closed, we have $A+B=k+2p_{\epsilon_i}+2k\mathbb{Z}\subseteq Z_{2\epsilon_i}.$
	    We see that $Z_{2\epsilon_i}= k+2p_i+2k\mathbb{Z}$, since $\Psi$ is closed and we have $Z_{2\epsilon_i}+Z_{-\epsilon_i}\subseteq Z_{\epsilon_i}.$  
	    This gives us that $\Psi=\Psi_k(p,I).$
	\end{proof}
	
	\begin{rem}
	As in the reduced case, we define $Gr(\Psi)$ is semi-closed if $\alpha,\beta\in Gr(\Psi)$ such that $\alpha+\beta\in \mathring{\Phi}$ but $\alpha+\beta\notin Gr(\Psi)$ implies that $(\alpha,\beta,\alpha+\beta)$ is of type $(s,s,d)$ or $(\ell,\ell,d).$ Note that the gradient of $\Psi_\tau(p,I,J)$ is semi-closed subset of $\mathring{\Phi}$, while the gradient of $\Psi_k(p,I)$ is closed in $\mathring{\Phi}.$
	\end{rem}
	
	\begin{rem}
	Maximal closed subroot systems in this case are classified in \cite[Section 9]{RV19}. Note that
	\begin{enumerate}
	    \item $\Psi_I(A_{2n}^{(2)})=\Psi_n(p,I,J)$ with $I=I_n,J=I,p_{\epsilon_i}=\frac{1}{2}, i\in I.$
	    \item $\Psi(p,n_s)=\Psi_n(p,I).$
	    \item $\widehat{A_J}$ is a union of a symmetric closed subset of $A_{2n-1}^{(2)}$ and $\Psi_n(p,I).$
	\end{enumerate}
	\end{rem}

	 \section{Real closed Subsets and Regular subalgebras}\label{regularsect}
	 In this section, we will study the correspondence between the symmetric closed subsets of $\Phi$ and the regular subalgebras of the affine Lie algebra $\lie g$ generated by them.
	  Let us denote by $\mathcal{C}_{\text{sym}}(\lie g)$ the set of symmetric closed subsets of $\Phi$ and recall that 
	 $\lie g(\Psi)$ denotes the subalgebra of $\lie g$ generated by $\bigcup_{\alpha\in \Psi}\lie g_\alpha$, for $\Psi\in \mathcal{C}_{\text{sym}}(\lie g)$.
	 Denote by $\mathcal{R}(\lie g)=\{\lie g(\Psi) : \Psi\in \mathcal{C}_{\text{sym}}(\lie g)\}$ the set of all regular subalgebras of $\lie g$ generated by the symmetric closed subsets of $\Phi,$ and 
	 define a map $$\iota_{\lie g}:\mathcal{C}_{\text{sym}}(\lie g)\to \mathcal{R}(\lie g)$$ by $\iota_{\lie g}(\Psi)=\lie g(\Psi)$.
	 The following result is well-known in the finite setting (see, for example \cite[Proposition 4.1]{DG20}):
	 \begin{prop}
     Let $(\mathring{\lie g}, \mathring{\lie h})$ be the pair of finite-dimensional simple Lie algebra and its Cartan subalgebra corresponding to the root system $\mathring{\Phi}$. Then 
     the map $\mathring{\Psi}\mapsto \lie g(\mathring{\Psi})$ from the set of closed subsets of $\mathring{\Phi}$ to
    the set of $\mathring{\lie h}$-invariant subalgebras of $\mathring{\lie g}$ is bijective. \qed
     \end{prop}
	 It is easy to see that unlike in the finite case, the map $\iota_{\lie g}$ is not injective for any $\lie g$, even if we restrict it to 
	 symmetric closed subsets of $\Phi.$

      \begin{example}
          Let $\lie{g}$ be any affine Lie algebra not of type $A_{2n}^{(2)}$ and $\alpha$ be a short root in $\mathring{\Phi}.$ Consider two symmetric, real closed subsets $\Psi_1,\Psi_2$ of $\Phi$ defined by $$\Psi_1:=\{\alpha+\delta,-\alpha+\delta,-\alpha-\delta,\alpha-\delta\},\ \ \Psi_2:=\{\alpha+3\delta,\alpha+\delta,-\alpha-3\delta,-\alpha-\delta\}.$$ The following relations show that $\pm 2\delta\in\Delta(\Psi_i),\ i=1,2:$  
          \begin{equation}\label{getimroots}
          \begin{aligned}
              \mathbb{C}\alpha^\vee\otimes t^{\pm 2}&=[\lie{g}_{\pm\alpha}\otimes t^{\pm 1},\lie{g}_{\mp \alpha}\otimes t^{\pm 1}]\subseteq \lie{g}(\Psi_1),\\
              \mathbb{C}\alpha^\vee\otimes t^{\pm 2}&=[\lie{g}_{\pm\alpha}\otimes t^{\pm 3},\lie{g}_{\mp \alpha}\otimes t^{\mp 1}]\subseteq \lie{g}(\Psi_2).
          \end{aligned}
          \end{equation}
          Since $[\alpha^\vee\otimes t^r,\lie{g}_{\pm\alpha}\otimes t^s]\neq 0$ for any $r,s\in \bz$, it follows that $\lie{g}_{\beta\pm 2\delta}\subseteq \lie{g}(\Psi_i)$ for all $\beta\in \Psi_i,\ i=1,2.$ Again using the commutation relations as in \cref{getimroots} we get $\mathbb{C}\alpha^\vee\otimes t^{\pm 4}\subseteq \lie{g}(\Psi_i), \ i=1,2.$ Proceeding in this way we get that $\lie{g}_{\beta\pm 2\bz\delta}\subseteq \lie{g}(\Psi_i)$ for all $\beta\in \Psi_i, \ i=1,2.$ Hence we have $$\lie{g}(\Psi_1)=\bigoplus_{r\in \bz}\lie{g}_{\pm \alpha+(2r+1)\delta}\oplus \bigoplus_{r\in\bz}\mathbb{C}\alpha^\vee\otimes t^{2r},\ \ \lie{g}(\Psi_2)=\bigoplus_{r\in \bz}\lie{g}_{\pm \alpha+(2r+3)\delta}\oplus \bigoplus_{r\in\bz}\mathbb{C}\alpha^\vee\otimes t^{2r}$$     
          So $\Psi_1\neq \Psi_2$ but $\iota_{\lie{g}}(\Psi_1)=\iota_{\lie{g}}(\Psi_2).$ 
          \end{example}
\begin{rem}
     It is possible to get similar examples using \cref{b2} when $Gr(\Psi)$ is of type $B_2$, e.g. $\Psi_1$ be given by \cref{b2}  and $\Psi_2$ be defined by switching the role of $a_1$ and $a_2.$ We will leave to the details, one can prove that $\Psi_1\neq \Psi_2$, but $\iota_{\lie{g}}(\Psi_1)=\iota_{\lie{g}}(\Psi_2).$
\end{rem}

\subsection{}	 The previous discussion motivates us to look for the best possible subclass of $\mathcal{C}_{\text{sym}}(\lie g)$ for which the map $\iota_{\lie g}$ restricted to that class
	 is injective.	 Recall a result from \cite[Corollary 11.1.5]{RV19}, the map $\Psi\mapsto \lie g(\Psi)$ is injective if we restrict to the
	  set of closed subroot systems of $\Phi$. That means the map $\iota_{\lie g}$ restricted to all closed subroot systems of $\Phi$ is injective.
	  We indeed prove that this is only the best possible subclass of symmetric closed subsets of $\Phi$ for which we can have injective:
\begin{prop}\label{injectivemaps}
Suppose $\mathcal{S}$ is a subclass of the set of symmetric closed subsets of $\Phi$ such that the restriction of $\iota_{\lie g}$ is injective, then $\mathcal{S}$ must be the set of all closed subroot systems of $\Phi$.
\end{prop}
\begin{proof}
    Let $\Psi$ be a symmetric closed subset of $\Phi,$ which is not a subroot system of $\Phi.$ Then by \cref{smallestclosed}, we know that
    $\Psi'=\Delta(\Psi)\cap \Phi$ is the minimal closed subroot system of $\Phi$ that contains $\Psi.$ It is clear that $\Psi\neq \Psi'$, but 
    $\lie{g}(\Psi)=\lie{g}(\Psi'),$ i.e.,
    $\iota_{\lie{g}}(\Psi)=\iota_{\lie{g}}(\Psi').$ So the largest subclass $\mathcal{S}$ of symmetric closed subsets of $\Phi$ for which the restriction 
    $\iota|_{\mathcal{S}}$ is injective is the set of all closed subroot systems of $\Phi$.
\end{proof}
    \begin{cor}
        Let $\Psi$ be a symmetric closed subset of $\Phi.$ Then $\Delta(\lie g(\Psi))\cap \Phi=\Psi$ if and only if $\Psi$ is a closed subroot system of $\Phi.$
    \end{cor}
    
\subsection{} Fix $\Psi\in \mathcal{C}_{\text{sym}}(\lie g)$, now we will determine the preimage $\iota^{-1}_{\lie g}(\lie g(\Psi))$ using our results, i.e., we will determine all possible 
$\Psi'\in \mathcal{C}_{\text{sym}}(\lie g)$ such that $\lie g(\Psi')=\lie g(\Psi).$ Note that $\Delta(\Psi)\cap\Phi$ is the unique real closed subroot system in $\iota_{\lie{g}}^{-1}(\lie{g}(\Psi))$, so it is enough to determine $\iota_{\lie{g}}^{-1}(\lie{g}(\Psi))$ when $\Psi$ is a real closed subroot system. If $\Psi'\in \iota_{\lie{g}}^{-1}(\lie{g}(\Psi)),$ then from the definition it is clear that $Gr(\Psi)=Gr(\Psi')$ and by \cref{smallestclosed} we have $$\iota_{\lie{g}}^{-1}(\lie{g}(\Psi))=\{\Psi'\in \mathcal{C}_{\text{sym}}(\lie g): \Delta(\Psi')\cap \Phi=\Psi\}.$$ To determine the preimage we need the following example.

\begin{example}\label{rootb2}
    Let $\Psi$ be a symmetric closed subset of $\Phi$ such that $Gr(\Psi)$ is of type $B_2.$ Then by \cref{prop:b2} there is a $\bz$-linear function $p:Gr(\Psi)\to \bz,n_\ell\in 2\bz_+$ and odd integers $1\le a_1,a_2\le n_\ell$ with $a_1+a_2\equiv 0(\mathrm{mod}\ n_\ell)$ such that $$\Psi^+=\{\alpha_i+ (p_i+A_i)\delta: i=1, 2\}\cup\{\pm(\alpha_1\pm \alpha_2)+(\pm p_1\pm p_2+n_\ell\bz)\delta:\alpha\in \mathring{\Phi}^+_\ell\}$$ where $A_i=n_\ell\bz\cup (a_i+n_\ell\bz), i=1,2.$ We shall determine $\Delta(\Psi)$ in this case and we shall show that $\Delta(\Psi)\cap \Phi$ is of the form \cref{b2subroot} with $n_s=\mathrm{gcd}(a_1,n)=\mathrm{gcd}(a_2,n).$ Since $a_1$ is odd, we have $2\nmid\ n_s.$ Since $a_1+a_2\equiv 0\ (\mathrm{mod}\ n_\ell)$ we have 
    \begin{align*}
        &Z_{\alpha_1}(\Psi)=n_\ell\bz\cup (a_1+n_\ell\bz), \ \ Z_{-\alpha_1}(\Psi)=n_\ell\bz\cup (a_2+n_\ell\bz)\  \text{ and }\\
        &Z_{\alpha_2}(\Psi)=n_\ell\bz\cup (a_2+n_\ell\bz), \ \ Z_{-\alpha_2}(\Psi)=n_\ell\bz\cup (a_1+n_\ell\bz)
    \end{align*} 
    Hence $Z_{\alpha_i}(\Psi)+Z_{-\alpha_i}(\Psi)=n_\ell\bz\cup (a_1+n_\ell\bz)\cup (a_2+n_\ell\bz)$ for $i=1,2.$ We shall show that $Z_{\pm\alpha_i}(\Delta(\Psi))\supseteq\pm p_i+\mathrm{gcd}(a_i,n_\ell)\bz.$ Since for any $r\in \bz,$ $[\alpha_1^\vee,\lie{g}_{\pm\alpha_1+r\delta}]\neq 0$ holds, for any $s\in \bz_+$ we have  $p_1+sa_1+n_\ell\bz\subseteq  Z_{\alpha_1}(\Delta(\Psi)).$  Also since $a_1+a_2\equiv 0\ (\mathrm{mod}\ n_\ell),$ $a_2+n_\ell\bz$ lies in the set $\sum_{k\in \bz}ka_1+n_\ell\bz.$ Exchanging the role of $a_1$ and $a_2$ in the argument we get $$Z_{\alpha_i}(\Delta(\Psi))\supseteq p_i+\sum_{k\in \bz}(ka_i+n_\ell\bz)=p_i+\mathrm{gcd}(a_1,n_\ell)\bz, \ \ i=1,2.$$ 
    Although $[\alpha_i^\vee,\lie{g}_{\alpha_1\pm\alpha_2+r\delta}]\neq 0, i=1,2,$ we have $p_1\pm p_2+sa_1+n_\ell\bz\subseteq  Z_{\alpha_1\pm\alpha_2}(\Delta(\Psi))$ if and only if $s$ is even. Hence it follows that $Z_{\alpha_1\pm\alpha_2}(\Delta(\Psi))\supseteq p_1\pm p_2+2\mathrm{gcd}(a_1,n_\ell)\bz.$ But $$\Psi':=\{\pm \alpha_i+\lceil\pm p_i+\mathrm{gcd}(a_1,n_\ell)\bz\rceil\delta\}\cup \{\pm(\alpha_1\pm \alpha_2)+\lceil\pm p_1\pm p_2+2\mathrm{gcd}(a_1,n_\ell)\bz\rceil\delta\}$$ is a real closed subroot system of $\Phi$ containing $\Psi$ and contained in $\Delta(\Psi)\cap \Phi.$  By \cref{smallestclosed} we have $\Psi'=\Delta(\Psi)\cap \Phi$ and hence the result.
\end{example}

\subsection{}
We first determine the preimage for $\Psi$, which is  irreducible and of the form \cref{b2subroot}.
\begin{lem}\label{preimageb2}
    Let $\Psi$ be an irreducible closed subroot system of $\Phi$ of the form \cref{b2subroot} and $n_s$ as given in \cref{b2subroot}. 
    For any given positive integer $r$, there are exactly $\varphi(2r)$ symmetric, real closed subsets  $\Psi'$ in $\iota_{\lie{g}}^{-1}(\lie{g}(\Psi))$ such that
    $n_\ell(\Psi')=2rn_s$,  where $\varphi$ is the Euler's totient function. 
    In particular, $\iota_{\lie{g}}^{-1}(\lie{g}(\Psi))$ is infinite in this case. For a fixed $n_s$ and $r,$ all $\Psi'\in \iota_{\lie{g}}^{-1}(\lie{g}(\Psi))$ with $n_\ell(\Psi')=2rn_s$ are given by \cref{b2} where $a_1$ is a cyclic generator of the group $\langle n_s \rangle$ in $\mathbb{Z}/(2rn_s)\bz.$

\end{lem}

\begin{proof}
If $\Psi'\in \iota_{\lie{g}}^{-1}(\lie{g}(\Psi))$, then $\Psi'$ must be of the form \cref{b2}. For a fixed $r,$  $\Psi'\in \iota_{\lie{g}}^{-1}(\lie{g}(\Psi))$ such that $n_\ell(\Psi')=2rn_s$ are in one to one correspondence with all $a_1$ such that $ 1\le a_1\le 2rn_s$ and $\mathrm{gcd}(a_1,2rn_s)=n_s$ by \cref{rootb2}. It is elementary to check that $\mathrm{gcd}(a_1,2rn_s)=n_s$ if and only if $a_1$ is a generator of the cyclic group $\langle n_s\rangle$ in $\bz/({2rn_s})\bz.$ Hence the number of such $a_1$ is equal to $\varphi(2rn_s/n_s)=\varphi(2r).$
\end{proof}
We are now ready to prove our main result of this section. 
\begin{prop}\label{preimage}
    For any irreducible real closed subroot system $\Psi$ of $\Phi,$ we have
    $$\iota_{\lie{g}}^{-1}(\lie{g}(\Psi))=\{\Psi\}$$ if either $Gr(\Psi) \text{ is not of type } A_1,B_2$ or $\text{ if } Gr(\Psi)=B_2 \text{ and }
         \Psi\text{ is as in } \cref{prop:equal1}.$ 
And $\iota_{\lie{g}}^{-1}(\lie{g}(\Psi))$ is infinite for all other cases.

\end{prop}

\begin{proof}
Suppose that no irreducible component of $Gr(\Psi)$ is of type $A_1$ or $B_2$, then any $\Psi'\in \iota_{\lie{g}}^{-1}(\lie{g}(\Psi))$ is a closed subroot system by our previous results and in this case we have $\iota_{\lie{g}}^{-1}(\lie{g}(\Psi))=\{\Psi\}.$  
Now suppose that $Gr(\Psi)$ is of type $B_2$.  From \cref{rootb2}, a symmetric, real closed subset $\Psi'$ of the form \cref{b2} generates a closed subroot system only of the form \cref{b2subroot}. Hence if $Gr(\Psi)$ is of type $B_2$ and $\Psi$ is of the form of \cref{prop:equal1}, then any $\Psi'\in \iota_{\lie{g}}^{-1}(\lie{g}(\Psi))$ is a closed subroot system of $\Phi$ and hence in this case $\iota_{\lie{g}}^{-1}(\lie{g}(\Psi))=\{\Psi\}.$

Now let $\Psi$ be of the form \cref{b2subroot} and let $n_s$ as in \cref{b2subroot}.
If $\Psi'\in \iota_{\lie{g}}^{-1}(\lie{g}(\Psi)),$ then $\Psi'$ is of the form \cref{b2} where $n_\ell(\Psi')$ is an even multiple of $n_s.$ Remaining follows from \cref{preimageb2}. 

\begin{rem}
     One of the drastic differences between the finite and affine root system theory is that even maximal symmetric closed subsets are not necessarily closed subroot systems. 
     For example, let $n_\ell$ be any even integer $\geq 4$ and $a_1$ be any integer which is $\le n_\ell$, $a_1\neq n_\ell/2$ and $\mathrm{gcd}(a_, n_\ell)=1.$ Then define $\Psi$ with this $n_\ell$ and $a_1$ as in \cref{prop:b2}. Then $\Psi$ is a proper symmetric closed subset which is not a closed subroot system and by \cref{rootb2}, $\Delta(\Psi)\cap \Phi=\Phi.$ 
     Suppose $\Psi'\supseteq \Psi$ is a symmetric closed subset that is not a subroot system, then
     we have that $Z_{\epsilon_i}(\Psi')$ is a union of two cosets of $n_\ell\bz$. This implies that $\Psi'=\Psi.$ Hence $\Psi$ is a maximal symmetric closed subset of $\Phi$, but not a closed subroot system.
\end{rem}

\end{proof}

\section{Summary}\label{summarysect}
Let $\Phi$ be a real affine root system and $\Psi$ be a symmetric closed subset of $\Phi$. In this section we shall summarize all the results in the following table. We assume that \cref{decomppsi} and \cref{decompgrpsi} are the decomposition of $\Psi$ and $Gr(\Psi)$ into irreducible components respectively. We assume that $p:Gr(\Psi)\to \bz$ is a $\bz$-linear function such that $p_\alpha\in Z_\alpha,\ \forall \alpha\in Gr(\Psi).$ We have three different type of forms for $\Psi.$
\begin{align}
    \Psi_i&=\{\alpha+(p_\alpha+n\bz)\delta:\alpha\in Gr(\Psi_i)\}\label{typeI}\\
    \Psi&=\{\alpha+(p_\alpha+n_s\bz)\delta:\alpha\in Gr(\Psi)_s\}\cup \{\gamma+(p_\gamma+mn_s\bz)\delta:\gamma\in Gr(\Psi)_\ell\}\label{typeII}\\
    \Psi_i^+&=\{\epsilon_i+ (p_{\epsilon_i}+A_i)\delta: i=1, 2\}\cup\{\alpha+(p_\alpha+n_\ell\bz)\delta:\alpha\in \mathring{\Phi}^+_\ell\}\text{ as in } \cref{prop:b2}\label{typeIII}
\end{align}

\cmark\, and \xmark\, in remark imply that $\Psi_i$ is a closed subroot system and is not a closed subroot system respectively.
\medskip

\begin{center}

\begin{NiceTabular}{cccc}[hvlines]
$Gr(\Psi_i)$  & Case  & Form of $\Psi_i$ &  Remark\\
\Block{3-1}{Closed} & $|\pi_m(Z_\alpha)|=1,\ \forall \alpha\in Gr(\Psi)$  &  \cref{typeI} & \cmark\\
 & \Block{2-1}{\begin{tabular}{c}
      $\exists \alpha\in Gr(\Psi)_s$  \\
      such that $|\pi_m(Z_\alpha)|>1$
 \end{tabular}}&  \cref{typeII} & \cmark\\
 & &  \cref{typeIII} & \xmark \\
 \Block{6-1}{Semi-closed} & \Block{2-1}{$D_{n+1}^{(2)}$} &  \cref{typeII} & \cmark \\
 & &  \cref{typeII} & \cmark \\
 & \Block{2-1}{$A_{2n-1}^{(2)}$} & \cref{typeIII} & \xmark \\
 & &  \cref{a22n-1} & \cmark \\
 & $D_{4}^{(3)}$ & \cref{typeI} & \cmark \\
 & $E_6^{(2)}$ & \cref{typeI} & \cmark\\
\Block{2-1}{} & \Block{2-1}{$A_{2n}^{(2)}$} &  \cref{a2n21} & \cmark \\
&  & \cref{a2n22} & \cmark \\
\end{NiceTabular}
\end{center}



\bibliographystyle{plain}
\bibliography{bibfile}

\end{document}